\documentclass[11pt,english]{article}
\usepackage{bbm}
\usepackage{amsthm}
\usepackage{amsmath}
\usepackage{amssymb}
\usepackage{euscript}
\usepackage{subfigure}
\usepackage{graphicx,caption}
\usepackage{geometry}
\usepackage{color}
\usepackage[bookmarksnumbered=true]{hyperref} 
\hypersetup{
colorlinks = true,
linkcolor = blue,
anchorcolor = blue,
citecolor = blue,
filecolor = blue,
urlcolor = blue
}

\newtheorem{theorem}{Theorem}
\newtheorem{proposition}[theorem]{Proposition}
\newtheorem{corollary}[theorem]{Corollary}
\newtheorem{lemma}{Lemma}
\theoremstyle{definition}
\newtheorem*{remark}{Remark}
\newtheorem{example}{Example}
\newtheorem{definition}{Definition}

\renewcommand{\epsilon}{\varepsilon}
\def\R{\mathbb{R}}
\def\N{\mathbb{N}}

\def\cH{\EuScript{H}}
\def\cM{\EuScript{M}}

\renewcommand{\epsilon}{\varepsilon}
\def\N{\mathbb{N}}
\def\R{\mathbb{R}}

\def\cA{\EuScript{A}}
\def\cB{\EuScript{B}}
\def\cF{\EuScript{F}}
\def\cG{\EuScript{G}}

\def\cM{\EuScript{M}}

\def\Id{\text{\rm Id}}

\begin{document}
	
\title{\textbf{A Liv\v{s}ic-type theorem and some regularity properties for nonadditive sequences of potentials}}
	
\author{Carllos Eduardo Holanda$^1$ \hspace{0.07cm} and \hspace{0.07cm} Eduardo Santana$^2$}
\date{%
$^1$\small{Department of Mathematics, Shantou University, Shantou, 515063, Guangdong, China\\
\textit{$^1$E-mail address:} \texttt{c.eduarddo@gmail.com} \\
$^2$Universidade Federal de Alagoas, Penedo, 57200-000, Alagoas, Brazil\\
\textit{$^2$E-mail address:} \texttt{jemsmath@gmail.com}
}}
	
\maketitle

\begin{abstract}
We study some notions of cohomology for asymptotically additive sequences and prove a Liv\v{s}ic-type result for almost additive sequences of potentials. As a consequence, we are able to characterize almost additive sequences based on their equilibrium measures and also show the existence of almost (and asymptotically) additive sequences of H\"older continuous functions satisfying the bounded variation condition (with a unique equilibrium measure) and which are not physically equivalent to any additive sequence generated by a H\"older continuous function. None of these examples were previously known, even in the case of full shifts of finite type. Moreover, we also use our main result to suggest a classification of almost additive sequences based on physical equivalence relations with respect to the classical additive setup.
\end{abstract}

\section{Introduction}

Let $X$ be a topological space and $T: X \to X$ a map. A sequence of functions $(f_n)_{n \ge 1}$ is \emph{asymptotically additive} with respect to $T$ if for each $\epsilon > 0$ there exists a function $f: X \to \R$ such that 
\[
\limsup_{n \to \infty}\frac{1}{n}\|f_n - S_nf\|_{\infty} < \epsilon,
\]
where $S_nf := \sum_{k=0}^{n-1}f \circ T^{k}$ is the \emph{additive sequence generated by the function $f$} and $\|\cdot\|_{\infty}$ is the usual supremum norm on the space of continuous functions. 

A sequence $\cF = (f_n)_{n \ge 1}$ is \emph{almost additive} with respect to $T$ if there exists $C > 0$ such that 
\[
-C + f_{m}(x) + f_{n}(T^m(x)) \le f_{m+n}(x) \le f_{m}(x) + f_{n}(T^m(x)) + C
\]
for every $x \in X$ and all $m,n \ge 1$. Feng and Huang showed in \cite{FH10} that almost additive sequences are in fact asymptotically additive. 

Inspired by some nomenclature in statistical mechanics and classical thermodynamic formalism (see for example \cite{Rue78}, \cite{VSF93} and \cite{Cun20}), we say that two nonadditive sequences of functions $\cF:= (f_n)_{n \ge 1}$ and $\cG:= (g_n)_{n \ge 1}$ (with respect to some given map) are \emph{physically equivalent}, or $\cF$ is \emph{physically equivalent} to $\cG$, if 
\[
\lim_{n \to \infty}\frac{1}{n}\|f_n - g_n\|_{\infty} = 0. 
\]

The following physical equivalence result was obtained in \cite{Cun20}: 

\begin{theorem}\label{NCT}
Let $\cF = (f_{n})_{n \in \N}$ be an asymptotically additive or almost additive sequence of continuous functions. Then, there exists a continuous function $f: X \to \mathbb{R}$ such that
\[
\lim_{n \to \infty} \frac{1}{n}\parallel f_{n} - S_{n} f \parallel_{\infty} = 0.
\]
\end{theorem}

A nonadditive sequence $\cF = (f_n)_{n \in \N}$ (with respect to $T: X \to X$) has \emph{bounded variation} if there exist $M > 0$ and $\epsilon > 0$ such that for $x, y \in X$ and $n \in \N$, we have that $d(T^{k}(x),T^{k}(y)) < \epsilon$ for every $k \in \{0,...,n-1\}$ implies $|f_n(x) - f_n(y)| \le M$ (see section~\ref{mr}). Observe that in Theorem \ref{NCT}, if the almost additive sequence $\cF$ has bounded variation, the additive sequence $(S_nf)_{n \in \N}$ does not necessarily have bounded variation in general.

A function $f: X \to \R$ is said to be \emph{Bowen} or to satisfy the \emph{Bowen property} if the additive sequence $(S_nf)_{n \in \N}$ has bounded variation (see section \ref{coh}).  

Let $X$ be a compact metric space and $T: X \to X$ be a continuous map. The following two questions were posted in \cite{Cun20}: 

\textbf{Question A.} \emph{Is there, given any almost additive sequence of continuous functions $\mathcal{F} = (f_n)_{n \in \N}$ with bounded variation, a continuous function $f: X \to \R$ such that $(S_nf)_{n \in \N}$ has bounded variation and}
\[
\lim _{n \to \infty}\frac{1}{n}\|f_{n} - S_{n}f\|_{\infty} = 0 \quad ?
\]

\textbf{Question B.} \emph{Is there, given any almost additive sequence $\mathcal{F} = (f_n)_{n \in \N}$ with bounded variation, a continuous function $f: X \to \R$ such that}
\[
\sup_{n \in \N} \| f_{n} - S_{n}f \|_{\infty} < \infty \quad ?
\]

Letting $\sigma: \Sigma^{\N} \to \Sigma^{\N}$ be the left-sided full shift of finite type, other natural finer questions about nonadditive regularity are the following (see also \cite{LLV22}):



\textbf{Question C.} \emph{Is there, given any almost additive sequence $\mathcal{F} = (f_n)_{n \in \N}$ of H\"older continuous functions with bounded variation (with respect to $\sigma$), a H\"older continuous function $f: X \to \R$ such that}
\[
\lim _{n \to \infty}\frac{1}{n}\|f_{n} - S_{n}f\|_{\infty} = 0 \quad ?
\]

\textbf{Question D.} \emph{Is there, given any asymptotically additive sequence $\mathcal{F} = (f_n)_{n \in \N}$ of (H\"older) continuous functions  with bounded variation (with respect to $\sigma$) and admitting a unique equilibrium measure, a Bowen (or H\"older) continuous function $f: X \to \R$ such that}
\[
\lim _{n \to \infty}\frac{1}{n}\|f_{n} - S_{n}f\|_{\infty} = 0 \quad ?
\]

It is well known that bounded variation does not guarantee uniqueness of equilibrium measures for asymptotically additive sequences, even in the case of full shifts of finite type. Based on that, one can easily find asymptotically additive sequences of locally constant functions which are \emph{not} physically equivalent to any additive sequence generated by a H\"older (or Bowen) continuous function (see Remark 4.5 in \cite{Cun20}). On the other hand, the same equivalence problem might be different for asymptotically additive sequences with unique equilibrium measures.

Based on the equivalence possibility brought by Theorem \ref{NCT}, all these questions arise naturally considering the different levels of additive and nonadditive regularity associated. We note that for subshifts of finite type, expanding and hyperbolic maps in general, if $f: X \to \R$ is a H\"older continuous function then the additive sequence $(S_nf)_{n \in \N}$ satisfies the bounded variation property (see section \ref{mr}).

Observe that \textbf{Question A} does not necessarily imply \textbf{Question B} in general. Moreover, if $\cF$ is uniformly bounded, that is, $\sup_{n \in \N} \| f_{n}\|_{\infty} < \infty$, the sequence has bounded variation and the questions \textbf{A} and \textbf{B} are always affirmatively answered taking any continuous function $f$ cohomologous to zero.

Let us now see some consequences of answering questions \textbf{A}, \textbf{B}, \textbf{C} and \textbf{D}. 

\begin{enumerate}
\item \emph{Uniqueness of equilibrium measures}: a positive answer to \textbf{Question A} implies that one can obtain the uniqueness of equilibrium measures for almost additive sequences with bounded variation directly from the classical uniqueness result for a single potential (see \cite{Bow75}, \cite{Bar06} and \cite{Mum06}).

\item \emph{Quasi-Bernoulli and Gibbs measures:} a positive answer to \textbf{Question B} readily implies that every quasi-Bernoulli measure is in fact a Gibbs measure with respect to some continuous potential (see subsection \ref{BHR} and also \cite{BM}, \cite{Cun20}). 


\item \emph{Nonadditive ergodic optimization}: giving a positive answer to \textbf{Question C} and \textbf{Question~D} would allow us to use ergodic optimization results for H\"older potentials to either simplify or automatically extend some results for almost and asymptotically additive sequences of potentials (see for example \cite{CLT01}, \cite{Jen06}, \cite{Mor08}, \cite{CH10}, \cite{Con16}, \cite{Boc18}, \cite{Jen19}, \cite{Zha19} and \cite{BHVZ21}). We believe that this would give a better understanding of more general nonadditive settings, such as the asymptotically subadditive setup (see for example \cite{GG16}). 

\item \emph{Regularity of the nonadditive topological pressure}: an affirmative answer to \textbf{Question~C} and \textbf{Question D} automatically gives important classes of asymptotically and almost additive sequences $\cF := (f_n)_{n \in \N}$ where the nonadditive topological pressure $t \mapsto P(t\cF)$ is actually analytic, instead of having $C^{1}$ regularity (or less) in general (see \cite{Rue78} and \cite{BD09}). 

\item \emph{Nonadditive multifractal analysis:} answering positively \textbf{Question C} and \textbf{Question~D} also immediately guarantees higher-regularity of the entropy and dimension spectra for important classes of asymptotically and almost additive sequences with respect to some dynamical systems with hyperbolic behavior (see \cite{BS01}, \cite{BSS02}, \cite{BD09} and \cite{BCW13}).    

\end{enumerate}

All these regularity issues indicate that, even with the information provided by Theorem~\ref{NCT}, the relationship between additive, asymptotically and almost additive sequences are not yet completely understood. Notice that if all these questions are positively answered, the almost (asymptotically) additive world and the additive world are in fact the "same" in a more complete sense, regarding many aspects of thermodynamic formalism, ergodic optimization and multifractal analysis for discrete-time dynamical systems. 

We note that a positive answer to \textbf{Question B} immediately implies a positive answer to \textbf{Question A}. Moreover, for some subshifts of finite type, and some types of expanding and hyperbolic setups, a positive answer to \textbf{Question C} (without the H\"older continuity hypotheses on the sequences) also implies a positive answer to \textbf{Question A}. 

In this work, we obtained the following characterization result for almost additive sequences of continuous functions:

\textbf{Main result} (Theorem \ref{BOU}). \emph{Let $T\colon X \to X$ be a continuous map on a compact metric space $X$, satisfying the Closing Lemma and having a point with dense orbit. Let $\cG = (g_n)_{n \in \N}$ be an almost additive sequence of continuous functions (with respect to $T$) with bounded variation. Then, the following are equivalent:}
\begin{enumerate}
\item $\lim_{n \to \infty}\|g_{n}\|_{\infty}/n = 0$;
	
\item $\sup_{n \in \N}\|g_n\|_{\infty} < \infty$;
	
\item there exists $K > 0$ such that  $|g_n(p)| \le K$ for all $p \in X$ and $n \in \N$ with $T^{n}(p) = p$. 
\end{enumerate}

This result is proved without using the physical equivalence given by Theorem \ref{NCT}, and it immediately shows that \textbf{Question A} is actually equivalent to \textbf{Question~B} for any topologically transitive map satisfying the Closing Lemma (see Corollary \ref{EQ}). This characterization result is an improvement concerning the understanding of regularity problems for nonadditive sequences, and it holds for setups including one-sided and two-sided full shifts of finite type, topologically transitive subshifts of finite type, repellers of topologically trasitive $C^{1}$ maps and locally maximal hyperbolic sets for topologically transitive $C^{1}$ diffeomorphisms. We also applied our main result to characterize almost additive sequences based on their equilibrium states and some cohomology relations, indicating that Theorem \ref{BOU} works like a nonadditive version of the classical Liv\v{s}ic theorem (\cite{Liv72}). Moreover, we show that Theorem \ref{BOU} is optimal in the sense that it does not hold for asymptotically and subadditive sequences in general. 

By using our main characterization result together with an example given by B\'ar\'any, K\"aenm\"aki and Morris in \cite{BKM} for the case of planar matrix cocycles, we are able to show the existence of almost additive sequences of H\"older continuous functions with bounded variation which are \emph{not} physically equivalent to any additive sequence generated by a H\"older continuous function. This result is new and gives a \emph{negative} answer to \textbf{Question~C}. Furthermore, we also show a construction giving a negative answer to \textbf{Question D} for the case of H\"older and Bowen regularity in general.  

Since a considerable part of the classical thermodynamic formalism, additive multifractal analysis and ergodic optimization for symbolic dynamics and some hyperbolic and uniformly expanding maps is mostly centered around H\"older continuous potentials, which is a natural class of functions to be considered in these setups (\cite{Rue78}, \cite{KH12}, \cite{VO16}, \cite{BS00}, \cite{BS01}. \cite{BSS02}, \cite{Boc18}, \cite{Jen19}), our negative answers to \textbf{Question C} and \textbf{Question D} go in the direction of revealing that almost (and asymptotically) additive sequences of potentials do not always possess the same expected regularity properties as the additive ones in general. Nevertheless, as far as we know, the more general Bowen regularity problem in \textbf{Question A} still remains open even in the case of full shifts of finite type.

The paper is organized as follows. Based on classical results and the physical equivalence theorem obtained in \cite{Cun20}, we study and compare some different aspects of cohomological notions for almost and asymptotically additive sequences. In the third section, we obtain our main result (Theorem \ref{BOU}) and we show how it is related with some results in the context of general matrix cocycles. As an application of our main theorem, we conclude the third section showing how to classify almost additive sequences based on cohomology and equilibrium measures. In the final section, we start with some examples showing setups where we can always answer affirmatively the \textbf{Questions} \textbf{A} and $\textbf{B}$, and setups where \textbf{Question A} is always positively answered but \textbf{Question B} cannot be affirmatively satisfied. In the next subsection, building on an example given in \cite{BKM}, we use again our main result to exhibit examples of almost additive sequences of H\"older continuous potentials satisfying the bounded variation property but  not physically equivalent to any additive sequence generated by any H\"older continuous potential, finally giving a negative answer to \textbf{Question C}. In the following subsection, inspired by some recent developments in the context of matrix cocycles, we use Theorem \ref{BOU} to suggest a classification of almost additive sequences based on the different types of physical equivalence relations with the additive setup. We conclude our work showing some constructions of almost and asymptotically additive sequences of different types, giving a negative answer to \textbf{Question D}, and also exploring some more general open problems concerning the Bowen and Walters regularity.

\section{Some notions of cohomology for asymptotically  additive sequences of functions}\label{coh}

Let $T: X \to X$ be a continuous map on a compact metric space $X$ and $(f_{n})_{n \in \N}$ be an \textit{almost additive} sequence of continuous functions (with respect to $T$), that is, $f_{n} : X \to \R$ is continuous for all $n \in \N$ and there exists $C>0$ such that
\[
- C + f_{n}(x) + f_{m}(T^{n}(x)) \leq f_{m + n}(x) \leq C+ f_{n}(x) + f_{m}(T^{n}(x)), 
\]
for all $x \in X$ and $m,n \in \N$. 

We say that a function $\varphi: X \to \R$ or the additive sequence $(S_n\varphi)_{n \in \N}$ satisfies the \emph{Walters property} or is a \emph{Walters function} if for each $\kappa > 0$ there exists $\epsilon > 0$ such that for $x, y \in X$ and $n \in \N$, we have that $d(T^{k}(x),T^{k}(y)) < \epsilon$ for every $k \in \{0,...,n-1\}$ implies $|S_n\phi(x) - S_n\phi(y)| < \kappa$. Moreover, we say that a function $\psi: X \to \R$ satisfies the \emph{Bowen property} or is a \emph{Bowen function} if there exist $M > 0$ and $\epsilon > 0$ such that for $x, y \in X$ and $n \in \N$, we have that $d(T^{k}(x),T^{k}(y)) < \epsilon$ for every $k \in \{0,...,n-1\}$ implies $|S_n\phi(x) - S_n\phi(y)| \le M$. It is clear from the definitions that every function satisfying the Walters property also satisfies the Bowen property. 

Let $T: X \to X$ be a continuous map on a compact metric space $X$. The following result is the \emph{Closing Lemma}, a classical well known tool in hyperbolic dynamics (see for example \cite{KH12}).   

\begin{lemma}\label{CCL} For every $\epsilon > 0$ there exists $\delta > 0$ such that if $x \in X$ and $n \in \N$ satisfying $d(T^n(x),x)<\delta$, then there exists $y \in X$ such that $T^n(y) = y$ and $d(T^k(x),T^k(y)) < \epsilon$ for all $0 \le k < n$.
\end{lemma}

A continuous function $f: X \to \R$ is said to be \emph{cohomologous to zero}, or a \emph{coboundary} (with respect to some continuous map $T$) when there exists a continuous function $q: X \to \R$ such that $f(x) = q(T(x)) - q(x)$ for all $x \in X$. We say that a function $f$ is cohomologous to another function $g$ if $f-g$ (or $g-f$) is a coboundary.

The following proposition is a more general version of the classical \emph{Liv\v{s}ic Theorem} originally obtained in \cite{Liv72}. 

\begin{proposition}\label{LIV}
Let $T\colon X \to X$ be a continuous map satisfying the Closing Lemma and having a point with dense orbit. Let $f\colon X \to \R$ be a continuous function satisfying the Walters property. Then $f$ is cohomologous to zero if and only if for every periodic point $x = T^n(x)$ we have $S_nf(x) = 0$.
\end{proposition}

Observe that when a potential $f$ is cohomologous to zero, by definition, there exists a continuous function $h\colon X \to \R$ such that $f = h\circ T - h$, and this readily implies that
\begin{equation}\label{ttt}
\int_{X}f d\mu = 0\quad\text{for all $T$-invariant measure $\mu$.} 
\end{equation}
Conversely, assume that condition~\eqref{ttt} holds. In particular, whenever $T^n(x) = x$ we also have $\int_{X}f d\nu = 0$ for the measure $\nu = \frac1n\sum_{k=0}^{n-1}\delta_{T^k(x)}$, and consequently
\[
0 = \int_{X}f d\nu = f(x) + f(T(x)) +\cdots + f(T^{n-1}(x)) = S_nf(x).
\]
Therefore, by the Liv\v{s}ic Theorem (Proposition \ref{LIV}) we conclude that $f$ is cohomologous to zero. We just obtained the following result:

\begin{corollary}\label{COH}
$f$ is cohomologous to zero if and only if $\int_{X}f d\mu = 0$ for every $T$-invariant measure $\mu$.
\end{corollary}

Let $\cM_T$ be the set of all $T$-invariant measures.

\begin{proposition}\label{EQC}
Under the conditions of Proposition~\ref{LIV}, $f$ is cohomologous to zero if and only if 
\[
\lim_{n \to \infty}\frac1n\|S_nf\|_{\infty}=0.
\]
In particular, 
\[
\lim_{n \to \infty}\frac1n\|S_nf\|_{\infty}=0 \textrm{ if and only if } \sup_{n \in \N}\|S_nf\|_{\infty} < \infty.
\]
\end{proposition}

\begin{proof}
Suppose we have $\lim_{n \to \infty}\frac1n\|S_nf\|_{\infty} = 0$. By the Lebesgue's dominated convergence theorem and the Birkhoff's ergodic theorem, we obtain
\[
0 = \int_{X}\lim_{n \to \infty}\frac1nS_nf d\mu = \int_{X}f d\mu
\]
for every $\mu \in \cM_T$. Hence, applying Corollary~\ref{COH} we conclude that $f$ is cohomologous to zero. Conversely, if $f$ is cohomologous to zero, there exists a continuous function $h$ such that $f(x) = h(T(x)) - h(x)$, which implies that $S_nf(x) = h(T^n(x)) - h(x)$ for every $n\in\N$. Consequently, $\|S_nf\|_{\infty} \le 2\|h\|_{\infty} < \infty$ for every $n\in\N$, and thus, $\lim_{n \to \infty}\frac1n\|S_nf\|_{\infty} =~0$.
\end{proof}

In this setting, Proposition~\ref{EQC} indicates a property for additive sequences that is equivalent to the notion of cohomology for functions. Actually, a notion of cohomology for asymptotically and almost additive sequences of continuous functions was introduced by Bomfim and Varandas in \cite{BV}. 

Let $X$ be a compact metric space and $T: X \to X$ be a continuous map. By the density of H\"older continuous functions on the space of continuous functions, one can show that for each asymptotically additive sequence of continuous functions $\cF = (f_n)_{n \in \N}$, there exists a family of H\"older continuous functions $(f_{\epsilon})_{\epsilon \in (0,1)}$ such that 
\[
\lim_{n \to \infty}\frac1n\|f_n - S_nf_{\epsilon}\|_{\infty} < \epsilon \quad \textrm{for any $\epsilon > 0$}
\]
(see Proposition 2.3 in \cite{BV}). The family $(f_{\epsilon})_{\epsilon \in (0,1)}$ is called an \emph{admissible family} for the sequence $\cF$.

\begin{definition}[{\cite[Definition~2.4]{BV}}]\label{bv}
Let $\cF$ be an asymptotically or almost additive sequence of functions with respect to a map $T: X \to X$. The sequence $\cF$ is said to be cohomologous to a constant if there exists an admissible family $(f_\epsilon)_{\epsilon > 0}$ for $\cF$ such that $f_{\epsilon}$ is
cohomologous to a constant for every small $\epsilon \in (0, 1)$, that is, there exists a constant $c_{\epsilon} \in \R$ and a
continuous function $u_{\epsilon}:X \to \R$ so that $f_{\epsilon} = u_{\epsilon} \circ T - u_{\epsilon} + c_{\epsilon}$. 
\end{definition}

Using this definition, the following result was obtained.

\begin{lemma}[{\cite[Lemma~2.5]{BV}}]\label{BVP}
An asymptotically additive sequence $\cF$ is cohomologous to a constant if and only if $(f_n/n)_{n \in \N}$ converges uniformly to a constant. In particular, $\cF$ is cohomologous to zero if and only if $\lim_{n \to \infty}\|f_n\|_{\infty}/n=0$.
\end{lemma}

Now, based on Theorem~\ref{NCT}, we can give a simpler definition of cohomology for asymptotically additive sequences.

\begin{definition}\label{CHM}	We say that an almost additive or asymptotically additive sequence of functions $\cF = (f_n)_{n \in \N}$ is cohomologous to a constant if there exists a continuous function $f$ cohomologous to a constant and such that
\[
\lim_{n \to \infty}\frac1n\|f_n - S_nf\|_{\infty}=0.
\]
\end{definition}
A reasonable thing to ask is if the two notions of cohomology are equivalent. In order to answer this, we have the following.

\begin{lemma}\label{CEQ}
Using Definition~\ref{CHM}, $\cF$ is cohomologous to a constant if and only if $\big(\frac{f_n}{n}\big)_{n \in \N}$ converges uniformly to a constant.
\end{lemma}

\begin{proof}
Suppose $\cF$ is physically equivalent to the additive sequence $(S_nf)_{n \in \N}$, where $f$ is cohomologous to a constant $c \in \R$, that is, there exists a continuous function $h\colon X \to \R$ such that $f - c = h\circ T - h$. This implies that $S_nf = h\circ T^n - h + cn$ for every $n\in\N$. Then,
\[
\begin{split}
\lim_{n \to \infty}\bigg\Vert\frac{f_n}{n} - c\bigg\Vert_{\infty} &= \lim_{n \to \infty}\bigg\Vert\frac{f_n}{n} - \frac{S_nf}{n} + \frac{h\circ T^n - h}{n}\bigg\Vert_{\infty}\\
& \le \lim_{n \to \infty}\frac1n\Vert f_n - S_nf\Vert_{\infty} + \lim_{n \to \infty}\frac1n\Vert h\circ T^n - h \Vert_{\infty} = 0
\end{split}
\]
Conversely, let the sequence $\big(\frac{f_n}{n}\big)_{n \in \N}$ converge to a constant $c \in \R$. Then
\[
\lim_{n \to \infty}\frac{1}{n}\|f_n - S_nc\|_{\infty} = \lim_{n \to \infty}\|\frac{1}{n}f_n - c\|_{\infty} = 0
\]	
and, by Definition \ref{CHM}, $\cF$ is cohomologous to a constant. 
\end{proof}

From Lemma~\ref{CEQ} together with Lemma~\ref{BVP}, we conclude that Definition~\ref{CHM} is, in fact, equivalent to Definition \ref{bv}. Based on this, we can say that an asymptotically (or almost) additive sequence of continuous functions $\cF = (f_n)_{n \in \N}$ is cohomologous to zero if and only if the sequence $(f_n/n)_{n \in \N}$ is uniformly convergent to zero. 

For the additive case, a continuous function $f: X \to \R$ is cohomologous to zero if and only if $\lim_{n \to \infty}\frac{1}{n}\|S_nf\|_{\infty} = 0$ under the conditions of Liv\v{s}ic Theorem. This implies that, in general, the classical definition of cohomology for a function is way stronger than the one we are suggesting for nonadditive sequences in Definition \ref{CHM}. On the other hand, Proposition \ref{EQC} also suggests a definition of cohomology for nonadditive sequences which is still weaker than the classical one but much stronger than Definition \ref{CHM}:

\begin{definition}\label{SCH}	We say that an almost additive or asymptotically additive sequence of functions $\cF = (f_n)_{n \in \N}$ is cohomologous to a constant if there exists a continuous function $f$ cohomologous to a constant and such that
\[
\sup_{n \in \N}\|f_n - S_nf\|_{\infty} < \infty.
\]
\end{definition}  
In this sense, $\cF$ is cohomologous to zero if and only if $\cF$ is uniformly bounded. 

In the next section, our main result gives a setup where the definitions \ref{bv}, \ref{CHM} and \ref{SCH} are equivalent for almost additive sequences (see Theorem \ref{BOU}). On the other hand, considering the same setup, Example \ref{asym} shows that definition \ref{SCH} is not compatible with definitions \ref{bv} and \ref{CHM} for asymptotically additive sequences. 

\section{A Liv\v{s}ic-type theorem for almost additive sequences}\label{mr}

Let $X$ be a compact metric space and $T: X \to X$ a continuous map. We say that a sequence of functions $\cF = (f_n)_{n \in \N}$ has \emph{bounded variation} if there exists $\epsilon > 0$ such that
\[
\sup_{n \in \N}\sup\{|f_n(x) - f_n(y)|: d_n(x,y) < \epsilon \} < \infty,
\]
where $d_n(x,y) := \max \{d(T^{k}(x), T^{k}(y)): 0 \le k \le n-1\}$. Moreover, we say that $\cF$ has \emph{tempered variation} if 
\[
\limsup_{\epsilon \to 0}\lim_{n \to \infty}\frac{\gamma_n(\epsilon)}{n} = 0,
\] 
where $\gamma_n(\epsilon):= \sup\{|f_n(x) - f_n(y)|: d_n(x,y) < \epsilon \}$.

We note that if $\phi: X \to \R$ satisfies the Bowen property then the additive sequence $(S_n\phi)_{n \in \N}$ has bounded variation. 

The Walters property for functions (additive sequences) also can be extended naturally to the nonadditive case. We say that a sequence of functions $(f_n)_{n \in \N}$ satisfies the \emph{Walters property} if for each $\kappa > 0$ there exists $\epsilon > 0$ such that for $x, y \in X$ and $n \in \N$, we have that $d(T^{k}(x),T^{k}(y)) < \epsilon$ for every $k \in \{0,...,n-1\}$ implies $|f_n(x) - f_n(y)| < \kappa$. 

As in the additive case, it is clear that a sequence satisfying the Walters property also satisfies the Bowen property. Moreover, for subshifts of finite type, uniformly expanding maps and hyperbolic sets for diffeomorphisms, the space of H\"older continuous functions is contained in the class of functions satisfying the Walters property (see for example \cite{Wal78}, \cite{Bou02} and Proposition 20.2.6 in \cite{KH12}).

A point $x \in X$ is said to be \emph{transitive} (with respect to a map $T$) if $\overline{\{T^{n}(x): n \in \N\}} = X$.  On the other hand, we say that a map $T: X \to X$ is \emph{topologically transitive} if for every pair of non-empty open subsets $U, V \subset X$ there exists $n \in \N$ such that $T^n(U) \cap V \neq \emptyset$. In particular, when $X$ is a compact metric space one can show that $T: X \to X$ admits a transitive point if and only if $T$ is topologically transitive (see for example \cite{Sil92}). 

The following theorem is our main result.
\begin{theorem}\label{BOU}
Let $T\colon X \to X$ be a continuous map  satisfying the Closing Lemma and having a point with dense orbit. Let $\cG = (g_n)_{n \in \N}$ be an almost additive sequence of continuous functions (with respect to $T$) with bounded variation. Then, the following are equivalent:
\begin{enumerate}
\item $\lim_{n \to \infty}\|g_{n}\|_{\infty}/n = 0$;

\item $\sup_{n \in \N}\|g_n\|_{\infty} < \infty$;

\item there exists $K > 0$ such that  $|g_n(p)| \le K$ for all $p \in X$ and $n \in \N$ with $T^{n}(p) = p$. 
\end{enumerate}
\end{theorem}

The following result is a direct consequence of Theorem \ref{BOU}.

\begin{corollary}\label{EQ}
Let $\cF = (f_n)_{n \in \N}$ be an almost additive sequence of continuous functions with respect to $T: X \to X$ and with bounded variation. Then, a continuous function $f: X \to \R$ such that $(S_nf)_{n \in \N}$ has bounded variation satisfies 
\[
\lim_{n \to \infty}\frac{1}{n}\|f_n - S_nf\|_{\infty} = 0 \quad \textrm{if and only if} \quad \sup_{n \in \N}\|f_n - S_nf\|_{\infty} < \infty.
\]
	
In particular, if $(S_nf)_{n \in \N}$ does not have bounded variation we have
\[
\sup_{n \in \N}\|f_n -S_nf\| = \infty.
\] 
\end{corollary}

Notice that Corollary \ref{EQ} readily implies that \textbf{Question A} is in fact equivalent to \textbf{Question B} for topologically transitive maps satisfying the Closing Lemma. We also note that Theorem \ref{BOU} is an extension of Proposition \ref{EQC} to the case of almost additive sequences having bounded variation. 

\begin{remark}
Observe that Proposition \ref{EQC} asks for the sequence of potentials to have the Walters property, which is stronger than the bounded variation condition. This is because the classical cohomology result obtained for a single potential is also stronger than the uniformly bounded one obtained in Theorem \ref{BOU}. As we shall see on section \ref{Liv-sec}, Theorem \ref{BOU} is also particularly related to Theorem 1.2 in \cite{Kal11}, where control over the periodic data implies control over the full data. 
\end{remark}

In order to prove Theorem \ref{BOU}, let us first obtain a key auxiliary result. 

\begin{lemma}\label{UBI}
Let $T\colon X \to X$ be a continuous map and let $\cH = (h_n)_{n \in \N}$ be an almost additive sequence of continuous functions with uniform constant $C > 0$ and such that $\lim_{n \to \infty}\| h_n \|_{\infty}/n = 0$. Then 
	
\begin{enumerate}
\item for every $k$-periodic point $x_0 \in X$, we have $\sup_{q \in \mathbb{N}}|h_{qk}(x_0)| \leq C$;
		
\item for every periodic point $x_0$, there exists a constant $L \ge 0$ (only depending on the period) such that $\sup_{n \in\N}|h_n(x_{0})| \le L$;

\item We have
\[
\sup_{\mu \in \cM_T}\bigg|\int_X h_{n} d\mu\bigg| \leq C \quad \textrm{for all $n \in \N$}. 
\]
\end{enumerate}
\end{lemma}

\begin{proof}
Since the sequence $\cH$ is almost additive, there exists $C>0$ such that
\[
h_n + h_n\circ T^{n} - C \le h_{2n} \le h_n + h_n\circ T^{n} + C
\]
for every $n\in\N$. One can use induction to show that
\begin{equation}\label{asfh}
\sum_{i=0}^{p-1}h_n(T^{ni}(x)) - (p-1)C \le h_{pn}(x) \le \sum_{i=0}^{p-1}h_n(T^{ni}(x)) + (p-1)C
\end{equation}
for every $n,p\in\N$ and $x \in X$.
	
Let us assume without loss of generality that $T^k(x_0) = x_0$ for some $k \in \N$. If $n = qk$ for some $q \in \N$, then $T^{ni}(x_{0}) = x_{0}$ and so
\[
\lim_{p \to \infty} \frac{1}{p}\sum_{i=0}^{p-1}h_n(T^{ni}(x_{0})) = h_n(x_{0}).
\]
On the other hand, since $\lim_{n \to \infty}\| h_n \|_{\infty}/n = 0$, we have
\[
\lim_{p \to \infty} \frac{h_{pn}(x)}{pn}=0
\]
and so it follows from \eqref{asfh} that
\begin{equation}\label{CC}
-C \le h_n(x_{0}) \le C,
\end{equation}
concluding the proof of item 1. 

Now consider the case $n = qk + r$, with $0 < r < k$. Let
\begin{equation}\label{AB}
A = \min\{h_{r}(x_{0}): r < k\} \quad \text{and} \quad B = \max\{h_{r}(x_{0}): r < k\}.
\end{equation}
By the almost additivity of the sequence $\cH$ together with \eqref{CC} and \eqref{AB}, we have
\[
-2C + A \le -C + h_{qk}(x_0) + h_r(T^{qk}(x_0)) \le h_n(x_0)
\]
and
\[
h_n(x_0) \le h_{qk}(x_0) + h_r(T^{qk}(x_0)) + C \le 2C + B.
\]	
Therefore,
\[
L_2:= \min\{A - 2C, -C\} \le h_n(x_{0}) \le \max\{B + 2C, C\}:= L_1
\]
for every $n\in\N$. Taking $L:= \max\{|L_1|,|L_2|\}$, the item 2 is proved.

Now let us prove item 3. If $\mu$ is a $T$-invariant probability, then it is also $T^{n}$-invariant for every $n \in \N$ and, by using that $\lim_{n \to \infty}\| h_{n} \|_{\infty}/n = 0$ in the inequalities \ref{asfh} together with the Birkhoff's ergodic theorem, we obtain that
\[
\displaystyle -C \leq \int_{X} \lim_{p \to \infty} \frac{1}{p}\sum_{i = 1}^{k -1} g_{n}(T^{in}(x)) d \mu(x) = \int_{X} g_{n} d \mu \leq C
\]
for every $n \in \N$. Since $\mu \in \cM_T$ is arbitrary, the lemma is proved. 
\end{proof}

\emph{Proof of Theorem \ref{BOU}.} 
Since $\cG$ has bounded variation, there exists $\epsilon> 0$ such that
\begin{equation}\label{BVA}
M:= \sup_{n \in \N}\sup\{|g_n(x) - g_n(y)|: d_n(x,y) < \epsilon\} < \infty.
\end{equation}

Let us start proving that 3 implies 2. Suppose that there exists $K > 0$ such that  $|g_n(p)| \le K$ for all $p \in X$ and $n \in \N$ with $T^{n}(p) = p$. Let $\omega \in X$ be a transitive point and let $\delta > 0$ be the number given by the Closing Lemma (Lemma \ref{CCL}). Since the orbit of $\omega$ is dense, there exists a number $L(\omega,\delta) \in \N$ such that for all $x \in X$ and $n \in \N$ there exists some $k \in \{0,1,...,L(\omega,\delta)\}$ with $d(T^{n}(x),T^{k}(\omega)) < \delta$. Letting $n > L(\omega,\delta)$, in particular there exists $k' \in \{0,1,...,L(\omega,\delta)\}$ such that $d(T^{n}(\omega), T^{k'}(\omega)) < \delta$, that is, $d(T^{n-k'}(T^{k'}(\omega)), T^{k'}(\omega)) < \delta$. 

By the Closing Lemma,  there exists a point $p \in X$ with $T^{n-k'}(p) = p$ and such that $d_{n-k'}(T^{k'}(\omega),p) < \epsilon$. Now applying the bounded variation condition \eqref{BVA} we have that
\[
|g_{n-k'}(T^{k'}(\omega)) - g_{n-k'}(p)| \le M,
\]
which implies that  $|g_{n-k'}(T^{k'}(\omega))| \le M + |g_{n-k'}(p)| \le M + K$. This together with almost additivity gives that 
\[
\begin{split}
|g_n(\omega)| = |g_{(n-k')+k'}(\omega)| &\le |g_{k'}(\omega)| + |g_{n-k'}(T^{k'}(\omega))| + C \\ 
& \le \max_{k \in \{0,1,...,L(\omega,\delta)\}}|g_k(\omega)| + M + K + C := K'.
\end{split}
\]
Since $n > L(\omega,\delta)$ was arbitrary, we have $|g_n(\omega)| \le K'$ for all $n \in \N$. By using almost additivity again, we also have that
\[
|g_n(T^{k}(\omega)) + g_{k}(\omega) - g_{n+k}(\omega)| \le C \quad \textrm{for all $n, k \in \N$},
\]
which gives 
\[
|g_n(T^{k}(\omega))| \le |g_k(\omega)| + |g_{n+k}(\omega)| + C \le 2K'+C  \quad \textrm{for all $n, k \in \N$}.
\]
Now let $x \in X$. Since $\omega$ is transitive, there exists a sequence $(\omega_q)_{q \ge 1} \subset \{T^{n}(\omega): n \in \N\}$ such that $\lim_{q \to \infty}\omega_q = x$. Since every function $g_n$ is continuous, we obtain that
\[
|g_n(x)| = \lim_{q \to \infty}|g_n(\omega_q)| \le 2K' + C.
\]

Therefore, by the arbitrariness of $x$, we conclude that $\sup_{n \in \N}\|g_n\|_{\infty} \le 2K'+ C < \infty$ as desired. It is immediate that 2 implies 1. By Lemma \ref{UBI}, it follows that 1 implies 3 and the theorem is proved.  \qed

Theorem \ref{BOU} does not hold for asymptotically additive nor subadditive sequences in general. In order to illustrate this, we give the following simple example.

\begin{example}\label{asym}
Let $T:X \to X$ be a continuous map and let $\cF = (f_n)_{n \in \N}$ be the sequence given by $f_n(x) := \sqrt n$ for every $n\in\N$ and every $x \in X$. It is clear that $\cF$ has bounded variation and is asymptotically additive and also subadditive with respect to $T$. Moreover, we have that $\lim_{n \to \infty}\|f_n\|_{\infty}/n = 0$ but $\sup_{n \in \N}\|f_n\|_{\infty} = \infty$. Actually, we also have
\[
\sup_{n\in\N}\|f_n - S_nf\|_{\infty} = \infty \quad \textrm{for every function $f\colon X \to \R$}.
\]
 
In fact, let us suppose by contradiction that there exist a continuous function $f\colon X \to \R$ and $L > 0$ such that $\sup_{n\in\N} \|f_n - S_nf \|_{\infty} \le L$. Given $x \in X$, we obtain
\[
\sup_{n\in\N} |f_n(x) - S_nf(x)| = \sup_{n\in\N}|\sqrt{n} - S_nf(x)| \le L.
\]
Setting $a_n := f(T^{n-1}(x))$, we have
\begin{equation}\label{SQT}
\sqrt{n} - L \le a_{1} + \dots + a_n \le \sqrt{n} + L.
\end{equation}
Since $x$ is arbitrary, we obtain $\sqrt{n} - L \le a_{k + 1} + \dots + a_{k + n} \le \sqrt{n} + L$ for every $k \in \N$.
Moreover, since these inequalities hold for every sequence with length $n$, we have
\begin{equation}\label{PRT}
p(\sqrt{n} - L) \le a_{1} + \dots + a_{pn} \le p(\sqrt{n} + L)
\end{equation}
for every $p \in \N$.
On the other hand, it follows directly from \eqref{SQT} that
\begin{equation}\label{PNT}
(\sqrt{pn} - L) \le a_{1} + \dots + a_{pn} \le (\sqrt{pn} + L) \quad \textrm{for every $p \in \N$}.
\end{equation}

Comparing \eqref{PRT} with \eqref{PNT} we obtain $p(\sqrt{n} - L) \le \sqrt{pn} + L$, which implies that
\[
\sqrt{n} - L \le \frac{\sqrt{n}}{\sqrt{p}} + \frac{L}{p} \quad \textrm{for every $p \in \N$.}
\]
By letting $p \to \infty$ we obtain $\sqrt{n} - L \le 0$ for every $n \in \N$, which is clearly a contradiction. Therefore, $\sup_{n\in\N}\|f_n - S_nf\|_{\infty} = \infty$ for every function $f \colon X \to \R$, as we claimed. \qed
\end{example}
Example \ref{asym} shows that the definition \ref{SCH} is not equivalent to definition \ref{CHM} for asymptotically additive sequences in general. Moreover, this also implies that Theorem \ref{BOU} has an optimal nonadditive setup in the sense that it cannot be extended to more general sequences, such as asymptotically additive and subadditive ones. 
 
\subsection{Matrix cocycles}\label{Liv-sec}

Let $X$ be a compact metric space and $T: X \to X$ a continuous map. Moreover, let $GL(d,\R)$ be the set of all invertible $d \times d$ matrices. A continuous map $\mathcal{A} : X \times \N \to GL(d,\R)$ is called a \emph{linear cocycle} over $T$ if for all $m,n\in\N$ and $x\in X$ we have:
\begin{enumerate}
\item $\cA(x,0) =$ Id;	
\item $\cA(x,n + m) = \cA(T^{m}(x),n)\cA(x,m)$.
\end{enumerate}
Every cocycle is generated by a function $A: X \to GL(d,\R)$, that is,
\[
\mathcal{A}(x,n) = A(T^{n-1}(x))...A(T(x))A(x) \quad \textrm{for all $n \in \N$ and $x \in X$}.
\]

Assume that the cocycle $\cA$ is generated by a continuous function $A: X \to GL(d,\R)$ which takes values in the set of matrices $d \times d$ with strictly positive entries. Moreover, we consider the pseudo-norm on $GL(d,\R)$ defined by $\|A\| = \sum_{i,j=1}^d |a_{ij}|$, denoting by $a_{ij}$ the entries of~$A$.

Now we consider the sequence of functions $\cF_{A}:= (a_n)_{n \in  \N}$ defined by
\begin{equation}\label{FAA}
a_n(x)= \log\|\cA(x,n)\| \quad \textrm{for all $n \in \N$ and $x \in X$}. 
\end{equation}

It follows from Lemma 2.1 in \cite{FL02} that the sequence $\cF_{A}$ is almost additive with respect to $T$.

Extending the definitions for derivative cocycles in \cite{BG06}, we say that the cocycle $\cA$ has \emph{tempered distortion} if for some $\epsilon > 0$
\[
\limsup_{n \to \infty}\frac{1}{n}\log \sup\left\{\|\cA(x,n) \cA(y,n)^{-1}\|: z \in M \ \text{and} \ x, y \in B_{n}(z,\epsilon)\right\} = 0
\]
Moreover, $\cA$ is said to have \emph{bounded distortion} if 
\[
\sup\left\{\|\cA(x,n) \cA(y,n)^{-1}\|: z \in M \ \text{and} \ x, y \in B_{n}(z,\epsilon)\right\} < \infty
\]
for some $\epsilon > 0$. It is clear that bounded distortion implies tempered distortion. 

Now observe that
\[
\|\cA(x,n) \cA(x,n)^{-1}\| = \|\Id\| = d
\]
for every $(x,n) \in X\times \N$, which implies that
\[
\|\cA(x,n)^{-1}\| \ge d \|\cA(x,n)\|^{-1}.
\]
Then
\begin{equation}\label{FL}
\|\cA(x,n) \cA(y,n)^{-1}\| 
\ge \frac{M}{d}\|\cA(x,n)\|\cdot\|\cA(y,n)^{-1}\| \ge M\|\cA(x,n)\|\cdot\|\cA(y,n)\|^{-1},
\end{equation}
where $M > 0$ is the constant satisfying
\[
1 \ge \frac{\min_{(i,j)} a_{ij}}{\max_{(i,j)} a_{ij}} \ge M,
\]
which exists because, by assumption, the entries of $A$ are all strictly positive.

So, it follows directly from \eqref{FL} that

\[
\bigl\lvert\log \|\cA(x,n)\| - \log \|\cA(y,n)\|\bigr\rvert \le - \log M + \log \|\cA(x,n) \cA(y,n)^{-1}\|.
\]
In particular, for $z \in X$ and $\epsilon > 0$ we have
\[
\sup_{x, y \in B_n(z,\epsilon)}|a_n(x) - a_n(y)| \le -\log M + \log \sup_{x, y \in B_n(z,\epsilon)}\|\cA(x,n) \cA(y,n)^{-1}\|.
\]
Hence, if the cocycle $\cA$ has tempered distortion, then the sequence $\cF_{A}$ has tempered variation, and if $\cA$ has bounded distortion, then $\cF_{A}$ has bounded variation. 

A cocycle $\mathcal{A} : X \times \N \to GL(d,\R)$ is said to satisfy \emph{domination} if there exist $C > 0$ and $\lambda > 0$ and a splitting $\R^{d} = E(x) \oplus F(x)$ such that 
\[
\|\mathcal{A}(x,n)u\| \ge Ce^{\lambda n} \|\mathcal{A}(x,n)v\|
\]
for all unitary vectors $u \in E(x)$ and $v \in F(x)$. For linear cocycles, one also can show that domination implies almost additivity of the sequence $\cF_{A}$ defined in $\eqref{FAA}$ (see for example \cite{BG09} and \cite{BKM} for more information on domination and almost additivity in the case of cocycles). 

Let $GL^{+}(d,\R) \subset GL(d,\R)$ be the set of all matrices with strictly positive entries.
We have the following application of Theorems \ref{BOU}, which is a particular version of one of the main results in \cite{Kal11}:

\begin{proposition}\label{KLN}
Let $T\colon X \to X$ be a topologically transitive continuous map on a compact metric space $X$ and satisfying the Closing Lemma. Let $A: X \to GL^{+}(d,\R)$ be a continuous function which generates a cocycle $\cA: X \times \N \to GL^{+}(d,\R)$ with bounded distortion or let $\cA: X \times \N \to GL(d,\R)$ be a cocycle satisfying domination with bounded distortion. Suppose there exists a compact set $\Omega \subset GL(d,\R)$ such that $\cA(p,n) \subset \Omega$ for all $p \in X$ and $n \in \N$ with $T^{n}(p) = p$. Then there exists a compact set $\widetilde{\Omega}$ such that  $\cA(x,n) \subset \widetilde{\Omega}$ for all $x \in X$ and $n \in \N$. 
\end{proposition}

\begin{proof}
By the hypotheses, $\cF_{A} = (a_n)_{n \in \N}$ given by $a_n(x):= \log\|\cA(x,n)\|$ is an almost additive sequence of continuous functions. Since the cocycle $\cA$ has bounded distortion, the sequence $\cF_{A}$ has bounded variation. Now suppose there exists a compact $\Omega \subset GL(d,\R)$ where $\cA(p,n) \subset \Omega$ for all $p \in X$ and $n \in \N$ with $T^{n}(p) = p$. Since the map $\cA(p,n) \mapsto \log\|\cA(p,n)\|$ is continuous, there exists $K > 0$ such that $a_n(p) \in [-K,K]$ for all $p \in X$ and $n \in \N$ with $T^{n}(p) = p$. By Theorem \ref{BOU}, there exists a constant $\widetilde{K} > 0$ such that $\sup_{n \in \N}\|a_n\|_{\infty} \le \widetilde{K}$. In particular, this implies that $e^{-\widetilde{K}} \le \|\cA(x,n)\| \le e^{\widetilde{K}}$ for all $x \in X$ and $n \in \N$. 

Therefore
\[
\|\cA(x,n) - \Id\| \le \|\cA(x,n)\| + \|\Id\| \le e^{\widetilde{K}} + d \quad \textrm{for all $x \in X$ and $n \in \N$},
\]
and the result is proved.
\end{proof}

Notice that when the function $A: X \to GL(d,\R)$ is H\"older continuous, the cocycle $\cA$ has in fact the bounded distortion property (see Proposition 5.1 in \cite{Kal11}). Then, Proposition \ref{KLN} is a particular version of Theorem 1.2 in \cite{Kal11} for the case of strictly positive cocycles (or cocycles with domination) satisfying the bounded distortion property.

\begin{remark}
For a general cocycle, the sequence $\cF_{A}$ defined in \eqref{FAA} is only subadditive. Following as in the proof of Proposition \ref{KLN}, one can check that Theorem 1.2 in \cite{Kal11} also gives an important class of subadditve sequences of continuous functions with bounded variation such that a uniform control over the periodic data implies a uniform control over all data. Based on this, one can naturally ask for which other classes of subadditive sequences this result can be extended. 
\end{remark}

\subsection{A characterization of almost additive sequences with the same equilibrium measures}

In this section, we will apply Theorem \ref{BOU} to see how we can identify sequences with the same equilibrium measures just based on the information about the periodic data of the system. 

Let $T\colon X \to X$ be a continuous map of a compact metric space, and let $\cF = (f_n)_{n \in \N}$ be an almost additive sequence of continuous functions with tempered variation (weaker than bounded variation condition). We have the following variational principle (see \cite{Bar06} and \cite{Mum06}: 
\[
P_T(\cF) = \sup_{\mu \in \cM_T}\bigg(h_{\mu}(T) + \lim_{n \to \infty}\frac1n\int_{X}f_n d\mu\bigg),
\]
where $P_T(\cF)$ is the nonadditive topological pressure of $\cF$ with respect to $T$ introduced in \cite{Bar96}, and $h_{\mu}(T)$ is the Kolmogorov-Sinai entropy. A measure $\nu \in \cM_T$ is said to be an \emph{equilibrium measure} for $\cF$ (with respect to $T$) if the supremum is attained in $\nu$, that is,
\[
P_T(\cF) = h_{\nu}(T) + \lim_{n \to \infty}\frac1n\int_{X}f_n d\nu.
\]
When $\cF$ is an additive sequence generated by a single continuous function $f: X \to \R$, one can easily see that
\[
P_T(\cF) = P_T(f) \quad \textrm{and} \quad \lim_{n \to \infty}\frac1n\int_{X}f_n d\mu = \int_{X}fd\mu \quad \textrm{for all $\mu \in \cM_T$},
\]
where $P_T(f)$ is the classical (additive) topological pressure of $f$ with respect to the map $T$. In this case, we recover the classical notions of topological pressure, variational principle and equilibrium measures for functions. For asymptotically and almost additive sequences, Theorem \ref{NCT} allows us to obtain the notion of nonadditive topological pressure and variational principle directly from the classical theory. 

It was introduced in \cite{Bar06} a definition of Gibbs measures with respect to almost additive sequences using Markov partitions. One also can define Gibbs measures with respect to almost additive sequences in a more general way, without the requirement of Markov partitions. We say that $\mu$ (not necessarily $T$-invariant) is a \emph{Gibbs measure} with respect to $\cF$ if for each $\epsilon > 0$ there exists a constant $K = K(\epsilon) \ge 1$ such that
\[
K^{-1} \le \frac{\mu(B_n(x,\epsilon))}{\exp[-nP_T(\cF) + f_n(x)]}\le K
\]
for all $x \in X$ and $n\in\N$, where $B_n(x,\epsilon) := \{y \in X: d_n(x,y) < \epsilon\}$. These two notions of Gibbs measures are equivalent when the system admits Markov partitions with arbitrarily small diameter, as in the case of locally maximal hyperbolic sets or repellers for $C^{1}$ diffeomorphisms (see for example \cite{Bow75a}). 

In the case of shifts $\sigma: \Sigma^{\N} \to \Sigma^{\N}$, the definition is simpler. A measure $\mu$ on $\Sigma^{\N}$ is said to be \emph{Gibbs} with respect to $\cF = (f_n)_{n \in \N}$ when there exists a constant $K \ge 1$ such that
\begin{equation}\label{NAG}
K^{-1} \le \frac{\mu_n(C_{i_1 \ldots i_n})}{\exp(-nP_{\sigma}(\cF) + f_n(x))} \le K
\end{equation}

for all $x \in C_{i_1 \ldots i_n}$ and $n \in \N$, where $C_{i_1\ldots i_n}$ is the set 
\[
C_{i_1 \ldots i_n}:= \{y = (j_1 j_2 \cdots) \in \Sigma^{\N}: j_1 = i_1,\ldots, j_n = i_n\}.
\]
These definitions are natural nonadditive versions of the classical Gibbs definitions with respect to a single continuous function (see definition \eqref{GIB}). 

Now let's recall the definition of uniformly expanding maps and repellers. Let $M$ be a Riemannian manifold, $T:M \to M$ a $C^{1}$ map, and let $\Lambda \subset M$ be a compact $T$-invariant set, that is, $T^{-1}(\Lambda) = \Lambda$. The map $T$ is said to be \emph{uniformly expanding} on $\Lambda$ if there exist constants $c > 0$ and $\lambda > 1$ such that 
\[
\|d_xf^{n}v\| \ge c\lambda^{n}\|v\| \quad \textrm{for all $x \in \Lambda$, $n \in \N$ and $v \in T_xM$}.
\]
In this case, the set $\Lambda$ is called a \emph{repeller} of $T$.  The following result is a criteria for uniqueness of equilibrium measures for nonadditive sequences.

\begin{proposition}[{\cite[Theorem~5]{Bar06}}]\label{BAR}
Let $\Lambda$ be a repeller of a $C^1$ topologically mixing map $T\colon \Lambda \to~\Lambda$, and let $\cF = (f_n)_{n \in \N}$ be an almost additive sequence of continuous functions on $\Lambda$ with bounded variation. Then $\cF$ admits a unique equilibrium measure, which also satisfies the Gibbs property with respect to $\cF$. 
\end{proposition}

Proposition \ref{BAR} also holds in the case of full shifts of finite type, topologically mixing subshifts of finite type and locally maximal hyperbolic sets for topologically mixing $C^{1}$ diffeomorphisms (see also Theorem 6 in \cite{Mum06}).

In the following result, we use Theorem \ref{BOU} to obtain a characterization of almost additive sequences based on equilibrium measures.
\begin{theorem}\label{CEM}
Let $\Lambda$ be a repeller of a $C^1$ topologically mixing map $T\colon \Lambda \to~\Lambda$. Let $\cF = (f_n)_{n \in \N}$ and $\cG = (g_n)_{n \in \N}$ be two almost additive sequences of continuous functions with bounded variation. Then $\cF$ and $\cG$ have the same equilibrium measure if and only if there exists a constant $K > 0$ such that $|f_n(p)-g_n(p) -n(P(\cF)-P(\cG))| \le K$ for all $p \in X$ and $n \in \N$ with $T^{n}(p) = p$. 
\end{theorem}

In order to prove this, we first need the following general result:

\begin{lemma}\label{UBA}
Let $X$ be a compact metric space, $T: X \to X$ a continuous map and let $\cF=(f_n)_{n \in \N}$ be an asymptotically additive sequence of continuous functions (with respect to $T$). Then
\[
\lim_{n \to \infty}\frac{1}{n}\|f_n\|_{\infty} = \sup_{\mu \in \cM_T} \bigg|\lim_{n \to \infty}\frac{1}{n}\int_{X}f_n d\mu\bigg|.
\]	
\end{lemma}
\begin{proof}
Proposition 2.1 in \cite{IP84} together with Lemma 2.2 in \cite{Cun20} gives that
\begin{equation}\label{aaa}
\lim_{n \to \infty}\frac{1}{n}\|S_nf\|_{\infty} = \sup_{\mu \in \cM_T} \bigg|\int_{X}f d\mu\bigg|
\end{equation}
for all continuous functions $f: X \to \R$. By Theorem \ref{NCT}, there exists a continuous function $f$ such that $(S_nf)_{n \in \N}$ is physically equivalent to $\cF$. Hence, in particular, 
\[
\lim_{n \to \infty}\frac{1}{n}\|f_n\|_{\infty} = \lim_{n \to \infty}\frac{1}{n}\|S_nf\|_{\infty} \quad \textrm{and} \quad \lim_{n \to \infty}\frac{1}{n}\int_{X}f_n d\mu  = \int_{X}f d\mu \quad \textrm{for all $\mu \in \cM_T$}.
\]
This together with \eqref{aaa} yields the desired result. 
\end{proof}

\emph{Proof of Theorem \ref{CEM}.}
If we suppose $|f_n(p)-g_n(p)-n(P(\cF) - P(\cG))| \le K$ for all $p \in X$ and $n \in \N$ with $T^{n}(p) = p$, Theorem \ref{BOU} gives that $\sup_{n \in \N}\|f_n - g_n -n(P(\cF) - P(\cG))\|_{\infty} < \infty$. By the definition of nonadditive topological pressure and Lemma \ref{UBA}, respectively, we have
\[
P(\cF) = P(\cH) \quad \textrm{and} \quad \lim_{n \to \infty}\frac{1}{n}\int_{X}f_n d\mu  =  \lim_{n \to \infty}\frac{1}{n}\int_{X}h_n d\mu \quad \textrm{for all $\mu \in \cM_T$}.
\]
where $\cH = (h_n)_{n \in \N}$ is the sequence $h_n:= g_n + n(P(\cF)-P(\cG))$. This readily implies that $\cF$ and $\cH$ have the same equilibrium measures. Since $\cH$ and $\cG$ also share the same equilibrium measure, the same is true for $\cF$ and $\cG$.
 	
Now let's prove the converse. By Proposition \ref{BAR}, the sequences $\cF$ and $\cG$ have unique equilibrium measures, and they satisfy the Gibbs property. Now suppose these equilibrium measures are the same unique measure $\nu \in \cM_T$. By the Gibbs property, for each $\epsilon > 0$ there exist constants $K_1 = K_1(\epsilon) \ge 1$ and $K_2 = K_2(\epsilon) \ge 1$ such that 
\[
K_1^{-1} \le \frac{\nu(B_n(x,\epsilon))}{\exp[-nP_T(\cF) + f_n(x)]}\le K_1 \quad \textrm{and} \quad K_2^{-1} \le \frac{\nu(B_n(x,\epsilon))}{\exp[-nP_T(\cG) + g_n(x)]}\le K_2
\]
for all $x \in X$ and $n\in\N$. This implies that 
\[
K_1^{-1}K_2^{-1} \le \exp[f_n(x)-g_n(x) - n(P_T(\cF)-P_T(\cG))]\le K_1K_2
\]
for all $x \in X$ and $n\in\N$. Then, 
\[
\|f_n - g_n - n(P_T(\cF)-P_T(\cG))\|_{\infty} \le \log(K_1K_2).
\]
In particular, we have that $|f_n(p)-g_n(p) -n(P_T(\cF)-P_T(\cG))| \le \log(K_1K_2)$ for all $p \in X$ and $n \in \N$ with $T^{n}(p) = p$, as desired.  \qed \\

\begin{remark}
By the equivalences in Theorem \ref{BOU}, the statement of Theorem \ref{CEM} also could be: $\cF$ and $\cG$ have the same equilibrium measure if and only if the sequence $\cF - \cG$ is cohomologous to the constant $P_T(\cF) - P_T(\cG)$, in the sense of definitions \ref{bv}, \ref{CHM} and \ref{SCH}. In this context, Theorem \ref{CEM} can be seen as a nonadditive counterpart of some classical additive results, for example, Theorem 12.2.3 in \cite{VO16} and Propositions 20.3.9 and 20.3.10 in \cite{KH12}.
\end{remark}

\begin{remark}
Based on the different versions of results concerning the uniqueness of equilibrium measures for almost additive sequences, Theorem \ref{CEM} also holds for topologically mixing subshifts of finite type, locally maximal hyperbolic sets for topologically mixing $C^{1}$ diffeomorphisms, and repellers of topologically mixing $C^{1}$ maps (see \cite{Bar06} and \cite{Mum06}). 
\end{remark}

\section{Regularity of almost and asymptotically additive sequences of continuous potentials}

In this section we will address the regularity problem for nonadditive sequences. We will start with some simpler cases in the non-hyperbolic context, then we will eventually attack and discuss the hyperbolic and related setups.  

\subsection{Non-hyperbolic setups}

In this subsection, we will consider maps where the periodic points are dominant and there is no transitive data, and setups where the transitive points are dominant and we have no periodic data. 

Let us now begin with an example of a setup where \textbf{Question A} and \textbf{Question B} can always be affirmatively answered. 

\begin{example}\label{RR} (Rational rotations)
Let $R_{\alpha}: \mathbb{T}^{n} \to \mathbb{T}^{n}$ be the rational rotation on the $n$-torus, that is, $R_{\alpha}(x) = x + \alpha \bmod 1$ with $\alpha \in \mathbb{Q}^{n}$. Considering $\cF := (f_n)_{n \in \N}$ any almost additive sequence of continuous functions with respect to $R_{\alpha}$, by Theorem \ref{NCT} there exists a continuous function $f: \mathbb{T}^{n} \to \R$ such that $\lim_{n \to \infty}\|f_n - S_nf\|_{\infty}/n = 0$. 
Since every point is periodic with the same period, by Lemma \ref{UBI} there exists a constant $L > 0$ such that $\sup_{n \in \N}\|f_n - S_nf\|_{\infty} \le L < \infty$. Notice that this uniform bound exists regardless of the sequence $\cF$ having bounded variation. Moreover, in this case, $\cF$ has bounded variation if and only if $(S_nf)_{n \in \N}$ has bounded variation. \qed
\end{example}

\begin{remark}
Example \ref{RR} is actually more general. In fact, by Lemma \ref{UBI} we obtain the same result for maps having at least a dense set of periodic points with a bounded period. This class of examples include periodic maps in general. 
\end{remark}

Let us now investigate what happens with some systems without periodic points. 

\begin{example}\label{UED} (Uniquely ergodic maps) Let $X$ be a compact metric space and consider $T: X \to X$ a uniquely ergodic map. Let $\cF:= (f_n)_{n \in \N}$ be an almost additive sequence of continuous functions with respect to $T$ and denote by $\nu$ the unique $T$-invariant measure. Then, in particular
\begin{equation}\label{UEM}
\lim_{n \to \infty}\frac{1}{n}\bigg\|S_n\phi - n\int_{X}\phi d\nu\bigg\|_{\infty} = 0 \quad \textrm{for all continuous functions $\phi: X \to \R$}.
\end{equation}
By Theorem \ref{NCT}, there exists a continuous function $f: X \to \R$ such that 
\[
\lim_{n \to \infty}\|f_n - S_nf\|_{\infty}/n = 0.
\]

This together with \eqref{UEM} yields
\[
\lim_{n \to \infty}\frac{1}{n}\bigg\|f_n - n\int_X f d\nu\bigg\|_{\infty} \le \lim_{n \to \infty}\frac{1}{n}\|f_n - S_nf\|_{\infty} + \lim_{n \to \infty}\frac{1}{n}\bigg\|S_nf - n\int_{X}f d\nu\bigg\|_{\infty} = 0.
\]
	
Hence, in the case of uniquely ergodic maps, there always exists a continuous function $f:~X \to~\R$ such that 
\[
\lim_{n \to \infty}\frac{1}{n}\bigg\|f_n - n\int_X f d\nu\bigg\|_{\infty} = 0,
\]
that is, the function given by Theorem \ref{NCT} can be taken as the constant $\int_{X}f d\nu$. 
	
Observe that the additive sequence $(S_n\int_X fd\nu)_{n \in \N}$ has bounded variation, and in this setup we can always answer \textbf{Question A} affirmatively. \qed  
\end{example}

The following classical result is due to Gottshalk and Hedlund (see for example Theorem 4.2 in \cite{FG21}). 

\begin{proposition}\label{GH55} Let $X$ be a compact metric space, $T: X \to X$ a homeomorphism such that every orbit is dense, and $\varphi: X \to \R$ a continuous function. Then the following assertions are equivalent:
\begin{enumerate}
\item there exists a continuous function $q: X \to \R$ such that $\varphi = q \circ T - q$;

\item there exists a point $x_0 \in X$ such that $\sup_{n \in \N}|S_n\varphi(x_0)| < \infty.$

\item $\sup_{n \in \N}\|S_n\varphi\|_{\infty} < \infty$. 
\end{enumerate}

\end{proposition}

In the following simple example, we show a setup where Theorem \ref{BOU} does not hold even in the additive case and with any strong regularity involved.  

\begin{example}\label{IR} (Irrational rotations) Let $R_{\alpha}: \mathbb{T} \to \mathbb{T}$ be the irrational rotation on the torus, that is, $R_{\alpha}(x) = x + \alpha \bmod 1$ with $\alpha \in \R/\mathbb{Q}$. The map $R_{\alpha}$ is a uniquely ergodic (and also minimal) homeomorphism. In addition, let $f:\mathbb{T} \to \R$ be a continuous function which is not cohomologous to any constant. It follows in particular from Example \ref{UED} that
\[
\lim_{n \to \infty}\frac{1}{n}\bigg\|S_nf - n\int_{\mathbb{T}}f d\nu\bigg\|_{\infty} = 0,
\]
where $\nu$ is the Lebesgue measure, that is, the unique $R_{\alpha}$-invariant measure on $\mathbb{T}$.

On the other hand, Proposition \ref{GH55} readily implies that $\sup_{n \in \N}\|S_n f - n \int_{\mathbb{T}}f d\nu\|_{\infty} = \infty$, regardless of the regularity of $f$.  
\qed
\end{example}

Observe that Example \ref{IR} also demonstrate that \textbf{Question B} cannot be positively answered for all setups in general.  

Even though the Examples \ref{RR}, \ref{UED} and \ref{IR} are interesting from a pure theoretical point of view, they are not as attractive as examples involving setups with some kind of hyperbolicity, where we can find the most strong applications in thermodynamic formalism and dimension theory in general. We will address the hyperbolic and related setups in the next sections. 

\subsection{H\"older regularity of almost additive sequences}\label{BHR}

For the sake of simplicity, in this section we are considering only the case of left-sided full shifts of finite type, but most of the results can be properly obtained for two-sided full shifts of finite type, topologically mixing subshifts of finite type, repellers of topologically mixing $C^{1}$ maps and locally maximal hyperbolic sets for topologically mixing $C^{1}$ diffeomorphisms.  

We start with a simple but useful auxiliary result.

\begin{lemma}\label{aas}
Let $X$ be a topological space and $T: X \to X$ a map. Let $\cA:= (a_n)_{n \in \N}$ and $\cB:= (b_n)_{n \in \N}$ be two sequences of functions such that $\sup_{n \in \N}\|a_n - b_n\| < \infty$. Then, $\cA$ is almost additive with respect to $T$ if and only if $\cB$ is almost additive with respect to $T$.   
\end{lemma}
\begin{proof}
Let $\xi_n:= a_n - b_n$ for all $n \in \N$. By hypotheses, there exists $L > 0$ such that $\sup_{n \in \N}\|\xi_n\|_{\infty} \le L$. Suppose $\cA$ is almost additive with respect to $T$ (and with constant $C > 0$). Then, for all $m,n \in \N$ we have
\[
\begin{split}
b_{n} + b_{m} \circ T^{n} - (3L + C) &= a_{n} - \xi_{n} + a_{m} \circ T^{n} - \xi_{m} \circ T^{n} - (3L + C)\\
&= a_{n} + a_{m} \circ T^{n} - C + (-\xi_{n} - L) + (-\xi_{m} \circ T^{n} - L) - L \\
&\le a_{n+m} - L \le a_{n+m} - \xi_{n+m} = b_{n+m}. 
\end{split}
\]
Analogously, we also have
\[
b_{n + m} \le b_{n} + b_{m} \circ T^{n} + (3L + C) \quad \textrm{for all $m, n \in \N$}.
\] 
Hence, the sequence $\cB$ is almost additive (with respect to $T$) and with constant $3L + C > 0$. The same argument works to prove the almost additivity of sequence $\cA$, and the result is proved.
\end{proof}

Let $\Sigma = \{1,\ldots,k\}$ be a finite alphabet and let $\Sigma^{\N}$ be the space of sequences $x = (i_1 i_2 \cdots)$ with $i_n \in \Sigma$ for every $n\in\N$. We define the \emph{left-sided full shift map} $\sigma\colon \Sigma^{\N} \to \Sigma^{\N}$ by $\sigma(i_1 i_2 i_3 \cdots) = (i_2 i_3 i_4 \cdots)$. Given a probability measure $\mu$ on $\Sigma^{\N}$, for any fixed $n$-tuple $(j_1,\ldots,j_n) \in \Sigma^n$ we write
\[\mu_n(C_{j_1 \ldots j_n}):= \mu(\{x = (i_1 i_2 \cdots) \in \Sigma^{\N}: i_1 = j_1,\ldots, i_n = j_n\}).
\]
A probability measure $\mu$ is said to be \emph{quasi-Bernoulli} if $\mu_1(C_{j}) > 0$ for every $j \in \Sigma$, and if there exists a constant $C \ge 1$ such that
\[
C^{-1}\le \frac{\mu_{n+m}(C_{j_1\cdots j_{n+m}})}{\mu_n(C_{j_1\cdots j_n})\mu_m(C_{j_{n+1}\cdots j_{n+m}})}\le C
\]
for every $m, n \ge 1$ and every $j_1,\ldots,j_{n+m} \in \Sigma$.  When one can take $C =1$, the measure $\mu$ is said to be \emph{Bernoulli}.  

Given $x = (i_1 i_2...) \in \Sigma^{\N}$ and $n \ge 1$, consider the set
\[
C_{i_1 \ldots i_n}(x):= \{y = (j_1 j_2 \cdots) \in \Sigma^{\N}: j_1 = i_1,\ldots, j_n = i_n\}.
\] 
Now given an arbitrary measure $\nu$ on $\Sigma^{\N}$, consider the sequence of continuous functions $\cF:= (f_n)_{n \in \N}$ given by $f_n(x):= \log\nu(C_{i_1 \ldots i_n}(x))$ for all $n \in \N$ and every $x \in \Sigma^{\N}$. It is easy to see that if $\nu$ is quasi-Bernoulli then $\cF$ is almost additive. In this case, we say that the sequence $\cF$ is generated by the quasi-Bernoulli measure $\nu$. 

Let $Bow(\Sigma^{\N},\sigma)$ be the vector space of continuous functions satisfying the Bowen property with respect to the shift $\sigma: \Sigma^{\N} \to \Sigma^{\N}$. 

Now lets us show how we can apply Theorem $\ref{BOU}$ to finally give a negative answer to \textbf{Question C}:

\begin{theorem}\label{NHD}
Let $\sigma: \Sigma^{\N} \to \Sigma^{\N}$ be the left-sided full shift. Then
	
\begin{itemize}
		
\item there exists an almost additive sequence of continuous functions generated by a quasi-Bernoulli measure on $\Sigma^{\N}$ and which is not physically equivalent to any H\"older continuous function.
\item there exist almost additive sequences of H\"older continuous functions with bounded variation and which are not physically equivalent to any additive sequence generated by a H\"older continuous function.
\end{itemize}
\end{theorem}

\begin{proof}
Let $\eta$ be a quasi-Bernoulli measure on $\Sigma^{\N}$ and consider the sequence $\cH = (h_n)_{n \in \N}$ generated by the measure $\eta$. Suppose there exist a constant $K > 0$ and a function $h \in Hol(\Sigma^{\N})$ such that $\|h_n - S_nh\|_{\infty} \le K < \infty$ for all $n\in\N$. This readily gives that 
\[
e^{-K}e^{S_nh(x)}\le \eta(C_{j_1 \cdots j_n}(x)) \le e^{S_nh(x)}e^{K}
\]
for all $x \in \Sigma^{\N}$ and $n\in\N$. Since $P_{\sigma}(h) = P_{\sigma}(\cH) = 0$ (see Theorem~2.1 in \cite{IY17}), we conclude that $\eta$ is a Gibbs measure with respect to the H\"older continuous function $h$.

Example 2.10 in \cite{BKM} shows a quasi-Bernoulli measure $\nu$ on $\Sigma^{\N}$ which is not Gibbs with respect to any H\"older continuous function. Now let $\cF = (f_n)_{n \in \N}$ be the sequence generated by $\nu$, that is, $f_n(x):= \log \nu (C_{i_1 \ldots i_n}(x))$ for all $x \in \Sigma^{\N}$ and $n \in \N$. This implies immediately that there is no $f \in Hol(\Sigma^{\N})$ such that 
\[
\sup_{n \in\N}\|f_n - S_nf\|_{\infty} < \infty,
\]
otherwise $\nu$ would be Gibbs with respect to the H\"older function $f$. Let $Hol(\Sigma^{\N})$ be the space of H\"older continuous functions. Since $Hol(\Sigma^{\N}) \subset  Bow(\Sigma^{\N},\sigma)$ and $\cF$ has bounded variation, Corollary \ref{EQ} guarantees that there is no H\"older continuous function $f$ such that 
\begin{equation}\label{HOL}
\lim_{n \to \infty}\frac{1}{n}\|f_n - S_nf\|_{\infty} = 0. 
\end{equation}
	
Moreover, since $Hol(\Sigma^{\N})$ is dense in the space of continuous functions on the compact space $\Sigma^{\N}$ (with respect to the sup norm $\|\cdot\|_{\infty}$), for each $n \ge 1$ there exists a H\"older continuous function $g_n: \Sigma^{\N} \to \R$ such that $\|f_n - g_n\|_{\infty} \le 1$. By Lemma \ref{aas}, the sequence $\cG := (g_n)_{n \in \N}$ is almost additive. It is also clear that $\cG$ has bounded variation. Moreover, for every continuous function $g: \Sigma^{\N} \to \R$ we have
\begin{equation}\label{LLV}
\lim_{n \to \infty}\frac{1}{n}\|f_n - S_ng\|_{\infty} = 0 \quad \textrm{if and only if} \quad \lim_{n \to \infty}\frac{1}{n}\|g_n - S_ng\|_{\infty}= 0.
\end{equation}
So, if there exists a H\"older continuous function $f$ such that 
\[
\lim_{n \to \infty}\frac{1}{n}\|g_n - S_nf\|_{\infty} = 0,
\]
it follows directly from \eqref{LLV} that the H\"older continuous function $f$ also satisfies \eqref{HOL}, which is a contradiction. Therefore, there is no additive sequence generated by a H\"older continuous function which is physically equivalent to the sequence $\cG$. 

Notice that by this construction, for each $\alpha > 0$ one can find an almost additive sequence of H\"older continuous functions $\cG_{\alpha}:= (g^{\alpha}_n)_{n \in \N}$ with $\|g^{\alpha}_n - f_n\|_{\infty} \le \alpha$ for all $n \in \N$, and not physically equivalent to any additive sequence generated by a H\"older continuous function, as desired.
\end{proof}

Theorem \ref{NHD} gives a negative answer to the problem raised in \cite{LLV22} (see Remark 3.4 in it). It is important emphasizing that these counter-examples of Theorem \ref{NHD} also demonstrate that, regarding the thermodynamic formalism, multifractal analysis and ergodic optimization, we are not always able to simplify or reduce the study of almost additive sequences with bounded variation to the additive case with H\"older (or Lipschitz) potentials, which are the most natural and most studied class of functions in the classical setup. In other words, the physical equivalence proved in \cite{Cun20} does not pass H\"older regularity in general. On the other hand, regarding more general regularity aspects, every almost additive sequence with bounded variation still could be physically equivalent to some additive sequence generated by a Bowen function.

\subsection{Bowen regularity of almost and asymptotically additive sequences}

In this section, based on physical equivalence associations, we study how one can relate more general regularity aspects of almost and asymptotically additive sequences to the regularity of single continuous functions. As we shall see, measures satisfying the Gibbs property and quasi-Bernoulli measures play an important role in our approach towards a more general treatment of regularity for sequences.

A probability measure $\mu$ on $\Sigma^{\N}$ is said to be \emph{Gibbs} with respect to a function $\phi: \Sigma^{\N} \to \R$ when there exists a constant $K \ge 1$ such that
\begin{equation}\label{GIB}
K^{-1} \le \frac{\mu_n(C_{j_1 \ldots j_n})}{\exp(-nP_{\sigma}(\phi) + S_n\phi(x))} \le K
\end{equation}

for all $x \in C_{j_1 \ldots j_n}$ and $n \ge 1$, where $P_{\sigma}(\phi)$ is the classical topological pressure of $\phi$ with respect to the shift $\sigma: \Sigma^{\N} \to \Sigma^{\N}$.

We note that every Gibbs measure with respect to a function is quasi-Bernoulli. In fact, let $\mu$ be a Gibbs measure with respect to a function $\phi$. By the definition, we have
\[
\frac{K^{-1}\exp[-(n+m)P_{\sigma}(\phi) + S_{n+m}\phi(x)]}{K^{2}\exp(-nP_{\sigma}(\phi) + S_n\phi(x))\exp(-mP_{\sigma}(\phi) + S_m\phi(\sigma^{n}(x)))} \le \frac{\mu_{n+m}(C_{j_1\cdots j_{n+m}})}{\mu_n(C_{j_1\cdots j_n})\mu_m(C_{j_{n+1}\cdots j_{n+m}})},
\]
\[
\frac{\mu_{n+m}(C_{j_1\cdots j_{n+m}})}{\mu_n(C_{j_1\cdots j_n})\mu_m(C_{j_{n+1}\cdots j_{n+m}})} \le \frac{K\exp[-(n+m)P_{\sigma}(\phi) + S_{n+m}\phi(x)]}{K^{-2}\exp(-nP_{\sigma}(\phi) + S_n\phi(x))\exp(-mP_{\sigma}(\phi) + S_m\phi(\sigma^{n}(x)))}.
\]
Since $S_{n+m}\phi = S_{n}\phi + S_{m}\phi \circ \sigma^{n}$, we finally obtain that 
\[
K^{-3} \le \frac{\mu_{n+m}(C_{j_1\cdots j_{n+m}})}{\mu_n(C_{j_1\cdots j_n})\mu_m(C_{j_{n+1}\cdots j_{n+m}})} \le K^{3} := \widetilde{K},
\]
as desired. This result can be extended to almost additive sequences:

\begin{proposition}\label{GQB}
Every Gibbs measure with respect to an almost additive sequence is quasi-Bernoulli.
\end{proposition}
\begin{proof}
Let $\cF = (f_n)_{n \in \N}$ be an almost additive sequence with Gibbs measure $\mu$ with uniform constant $K > 0$ (recall the definition in \eqref{NAG}). By almost additivity, there exists $C > 0$ such that $\|f_{n+m} - f_n - f_{m}\circ \sigma^{n}\|_{\infty} \le C$ for all $m,n \in \N$. Proceeding as in the additive case above, one can check that 
\[
(K^{3}e^{C})^{-1} \le \frac{\mu_{n+m}(C_{j_1\cdots j_{n+m}})}{\mu_n(C_{j_1\cdots j_n})\mu_m(C_{j_{n+1}\cdots j_{n+m}})} \le K^{3}e^{C}.
\]
Hence, $\mu$ is quasi-Bernoulli, as desired. 
\end{proof}

\begin{proposition}\label{CCH}
Let $\cF$ be a sequence of continuous functions with an equilibrium measure $\mu \in \cM_\sigma$ satisfying the Gibbs property with respect to $\cF$. Then, $\mu$ is quasi-Bernoulli if and only if $\cF$ is almost additive. 
\end{proposition}
\begin{proof}
By Proposition \ref{GQB} if $\mu$ is Gibbs with respect to $\cF$ then it is quasi-Bernoulli. For the converse just use the Gibbs property together with the fact that every sequence generated by a quasi-Bernoulli measure is almost additive. 	
\end{proof}

\begin{remark}
Proposition \ref{CCH} is, in particular, related to Theorem 2.8 in $\cite{BKM}$ for planar matrix cocycles in general. 
\end{remark}
The following result is a characterization of almost additive sequences for which we can answer affirmatively the \textbf{Question A} and the \textbf{Question B}.

\begin{theorem}\label{BKM-H}
Let $\cF = (f_n)_{n \in \N}$ be an almost additive sequence of continuous functions with respect to the left-sided full shift $\sigma: \Sigma^{\N} \to \Sigma^{\N}$ and satisfying the bounded variation property. Then, the following statements are equivalent:
\begin{enumerate}
\item the equilibrium measure of $\cF$ is Gibbs with respect to some continuous Bowen function;

\item there exists a continuous Bowen function $f: \Sigma^{\N} \to \R$ such that
\[
\lim_{n \to \infty}\frac{1}{n}\|f_n - S_nf\|_{\infty} = 0.
\]

\item there exists a continuous Bowen function $f: \Sigma^{\N} \to \R$ such that
\[
\sup_{n \in \N}\|f_n - S_nf\|_{\infty} < \infty.
\]
\end{enumerate} 
\end{theorem}

\begin{proof}
Let us start proving that 1 implies 3. Since $\cF$ has bounded variation, it follows by Theorem 6 in \cite{Mum06} (see also Theorem 5 in \cite{Bar06}) that $\cF$ has a unique equilibrium measure $\eta$, which is also Gibbs with respect to $\cF$. By assumption, $\eta$ is also Gibbs with respect to some continuous function $f$. Then, there exist constants $K_1 \ge 1$ and $K_2 \ge 1$ such that 
\begin{equation}\label{GBB}
K_1^{-1} \le \frac{\eta(C_{i_1 \ldots i_n}(x))}{\exp[-nP_{\sigma}(\cF) + f_n(x)]}\le K_1 \quad \textrm{and} \quad K_2^{-1} \le \frac{\eta(C_{i_1 \ldots i_n}(x))}{\exp[-nP_{\sigma}(\cF) + S_nf(x)]}\le K_2
\end{equation}
for all $x \in \Sigma^{\N}$ and $n \in \N$. This readily implies that $|f_n(x) - S_nf(x)| \le \log{K_1K_2}$ for all $x \in X$ and $n \in \N$, which yields 3. 

Conversely, suppose 3 holds. Then, there exists a constant $K_3 > 0$ and a continuous function $f \in Bow(\Sigma^{\N},\sigma)$ such that $\|f_n - S_nf\|_{\infty} \le K_3$ for all $n \in \N$. Hence, by the Gibbs property with respect to $\cF$ in \eqref{GBB}, we get that
\[
(K_1e^{K_3})^{-1}= K_1^{-1}e^{-K_3} \le \frac{\eta(C_{i_1 \ldots i_n}(x))}{\exp[-nP_{\sigma}(\cF) + S_nf(x)]}\le K_1e^{K_3}
\]
for all $x \in \Sigma^{\N}$ and $n \in \N$, as desired. Corollary \ref{EQ} immediately gives that items 2 and 3 are equivalent, and the result is proved. 
\end{proof}

\begin{remark}
We note that Theorem \ref{BKM-H} is related to the characterization result obtained in Theorem 2.9 in \cite{BKM} for planar matrix cocycles, but now also including Bowen functions in general and requiring only physical equivalence instead of the stronger hypotheses of uniformly bounded sequences.
\end{remark} 

We also know exactly what are the functions and almost additive sequences admitting measures satisfying the Gibbs property: 

\begin{proposition}\label{GBS}
Let $\sigma: \Sigma^{\N}$ be the left-sided full shift. Then
	
\begin{itemize}
\item $\phi \in Bow(\Sigma^{\N},\sigma)$ if and only if there exists a Gibbs measure with respect to $\phi$. 
		
\item  an almost additive sequence $\cF$ has bounded variation if and only if there exists a Gibbs measure with respect to $\cF$. 
\end{itemize} 
	
\end{proposition}
\begin{proof}
Let $\mu$ be a Gibbs measure with respect to a function $\phi$. By definition, there exists $K \ge 1$ such that
\[
K^{-2} \le \exp(S_n\phi(x) - S_n\phi(y)) \le K^{2}
\]
for all $x, y \in C_{i_1 \ldots i_n}$ and all $n \ge 1$. This implies that
\[
\sup_{n \in \N}\sup\{|S_n\phi(x) - S_n\phi(y)|: x, y \in C_{i_1 \ldots i_n}\} \le 2\log K  < \infty,
\]
and $\phi \in Bow(\Sigma^{\N}, \sigma)$. On the other hand, when $\phi \in Bow(\Sigma^{\N},\sigma)$ there exists an invariant measure which is the unique equilibrium state for $\phi$. Moreover, this measure is also Gibbs with respect to $\phi$ (see the classical works \cite{Bow75a} and \cite{Bow75}). 
	
For the nonadditive part, Theorem 6 in \cite{Mum06} (or Theorem 5 \cite{Bar06}) guarantees the existence of unique equilibrium measures satisfying the Gibbs property for almost additive sequences with bounded variation. The converse follows exactly as in the additive case. 
\end{proof}

Now let $T:X \to X$ be a continuous map, where $X$ is a compact metric space. Recall that a continuous function $f$ is said to be a coboundary with respect to a map $T: X \to X$ when there exists a continuous function $q$ such that $f(x) = q(T(x)) - q(x)$ for all $x \in X$. Following \cite{BJ02} and \cite{Bou21}, we say that a continuous function $f: X \to \R$ is a \emph{weak coboundary} if it is the uniform limit of coboudaries. It is well known that $f$ is a weak coboundary if and only if $\int_X f d\mu = 0$ for every $T$-invariant measure $\mu$ (see for example \cite{BJ02}, \cite{KR01}, and \cite{Kri71}). Based on this and on Lemma \ref{UBA}, $f$ is also said to be weak coboundary if and only if the sequence $(S_nf/n)_{n \in \N}$ converges uniformly to zero. 

We denote $Wcob(X,T)$ the vector space of weak coboundaries and $Cob(X,T)$ the vector space of coboundaries with respect to the map $T: X \to X$. In fact, one can check that every coboundary is a Walters function. On the other hand, weak coboundaries form a space that is strictly larger than the space of coboundaries in general (see for example \cite{BJ02} and \cite{Koc13}).

Given two subspaces  $U$ and $V$ in general, we define the new \emph{sum} subspace $U + V$ in the natural way, that is, $U + V = \{u + v: u \in U \textrm{ and } v \in V\}$. When we have a \emph{direct sum}, we shall write $U\oplus V$.

\begin{proposition}\label{QB-G}
Let $\sigma: \Sigma^{\N} \to \Sigma^{\N}$ be the left-sided full shift. For each continuous function $\varphi \in Wcob (\Sigma^{\N},\sigma) + Bow(\Sigma^{\N},\sigma)$ there exists an almost additive sequence of continuous functions $\cF:= (f_n)_{n \in \N}$ satisfying the Walters property and such that 
\[
\lim_{n \to \infty}\frac{1}{n}\|f_n - S_n\varphi\|_{\infty} = 0.
\] 
In particular, if $\varphi \in Bow(\Sigma^{\N},\sigma)$ then
\[
\sup_{n\in \N}\|f_n - S_n\varphi\|_{\infty} < \infty. 
\]
\end{proposition}
\begin{proof}
Since $\varphi \in Wcob(\Sigma^{\N},\sigma)+Bow(\Sigma^{\N},\sigma)$, there exist functions $u \in Wcob(\Sigma^{\N},\sigma)$ and $\xi \in Bow(\Sigma^{\N},\sigma)$ such that $\varphi = u + \xi$. Proposition \ref{GBS} guarantees the existence of a measure $\nu \in \cM_\sigma$ which is Gibbs with respect to $\xi$. In particular, there exists a constant $K \ge 1$ such that
\[
K^{-1} \le \frac{\nu(C_{i_1 \ldots i_n}(x))}{\exp(-nP_{\sigma}(\xi) + S_n\xi(x))} \le K
\]
for all $x \in \Sigma^{\N}$ and all $n \ge 1$. This implies that 
\begin{equation}\label{vcn}
|\log \nu(C_{i_1 \ldots i_n}(x)) + nP_{\sigma}(\xi) - S_n\xi(x)| \le \log K \quad \textrm{for all $x \in \Sigma^{\N}$ and all $n \ge 1$},
\end{equation}
Now define $\cG=(g_n)_{n \in \N}$ as the sequence given by $g_n(x):= \log \nu(C_{i_1 \ldots i_n}(x))$. Since every Gibbs measure is quasi-Bernoulli (see Proposition \ref{GQB}), the sequence of continuous functions $\cG$ generated by the quasi-Bernoulli measure $\nu$ is almost additive and satisfies the Walters property. Now consider the sequence $\cF =(f_n)_{n \in \N}$ given by $f_n:= g_n + nP(\xi)$. The sequence $\cF$ is clearly almost additive, satisfies the Walters property and it follows directly from \eqref{vcn} that 
\begin{equation}\label{bct}
\sup_{n \in \N}\|f_n - S_n\xi\|_{\infty} < \infty
\end{equation}
	
Moreover, since $\varphi - \xi \in Wcob(\Sigma^{\N},\sigma)$, we have that $(S_n\varphi)_{n \in \N}$ is physically equivalent to $(S_n\xi)_{n \in \N}$. This together with \eqref{bct} readily gives that 
\[
\lim_{n \to \infty}\frac{1}{n}\|f_n - S_n\varphi\|_{\infty} = 0,
\]
as desired. 
\end{proof}

In general terms, Proposition \ref{QB-G} is saying that every additive sequence generated by a Bowen continuous function with null topological pressure is physically equivalent to an almost additive sequence generated by some quasi-Bernoulli measure. 

In order to simplify the notation, let us denote $Hol:= Hol(\Sigma^{\N},\sigma)$, $Bow:= Bow(\Sigma^{\N},\sigma)$ and $Wcob(\Sigma^{\N},\sigma) = Wcob$. Moreover, let us define $Hof:= Hof(\Sigma^{\N},\sigma)$ as the subset of continuous functions which are not Bowen but have a unique equilibrium measure. We will refer to $Hof$ as the set of \emph{Hofbauer} functions (see \cite{Hof77}). 

Furthermore, consider the spaces $WB:= WB(\Sigma^{\N},\sigma) :=  Wcob(\Sigma^{\N},\sigma) + Bow(\Sigma^{\N}, \sigma)$ and $\widetilde{WB}:= C(\Sigma^{\N})/WB(\Sigma^{\N},\sigma)$, where $C(\Sigma^{\N})$ is the space of continuous functions on $\Sigma^{\N}$. We note that it is expected that $Hof \cap WB \neq \emptyset$ and $Hof \cap \widetilde{WB} \neq \emptyset$ (see for example \cite{PZ06}, \cite{Wal07}, \cite{Hu08}, \cite{IT12}, \cite{CT13}).   

Now let $f:\Sigma^{\N} \to \R$ be a continuous function. It is clear that if $(S_{n}f)_{n \in \N}$ is physically equivalent to some $\widetilde{f} \in Bow(\Sigma^{\N},\sigma)$ then we must have $f \in Wcob(\Sigma^{\N},\sigma)+ Bow(\Sigma^{\N},\sigma)$. Based on this, Proposition \ref{CCH}, Theorem \ref{BKM-H} and Proposition \ref{QB-G}, we can propose a classification of almost additive sequences with respect to the left-sided full shift (see Figure~\ref{AAS}). 

We separate the following types of almost additive sequences with respect to the left-sided full shift $\sigma: \Sigma^{\N} \to \Sigma^{\N}$:

\begin{itemize}
\item \textbf{Type 1.} Almost additive sequences $\cF_1$ with bounded variation and with the unique equilibrium measure satisfying the Gibbs property with respect to some function $\phi \in Bow$. They are always physically equivalent to $(S_n\psi)_{n \in \N}$ for some $\psi \in WB$ (i.e. $Bow$) and never physically equivalent to any additive sequence $(S_n\xi)_{n \in \N}$ with $\xi \in \widetilde{WB}$. 

\item \textbf{Type 2.} Almost additive sequences $\cF_2$ with bounded variation and with equilibrium measure not satisfying the Gibbs property with respect to any continuous function. In this case, the equilibrium is only weak-Gibbs with respect to some continuous function. These sequences are always physically equivalent to some $(S_n\psi)_{n \in \N}$ with $\psi \in Hof \cap\widetilde{WB}$. As a consequence, there is no $\xi \in WB$ such that $(S_n\xi)_{n \in \N}$ is physically equivalent to $\cF_2$. 

\item \textbf{Type 3.} Almost additive sequences $\cF_3$ without the bounded variation condition but admitting a unique equilibrium measure. They are  either physically equivalent to some $(S_n\psi_1)_{n \in \N}$ with $\psi_1 \in Hof \cap \widetilde{WB}$ or they are physically equivalent to some $(S_n\psi_2)_{n \in \N}$ with $\psi_2 \in Hof \cap WB$, which immediately implies $\psi_2 \in Bow$ (see Figure~\ref{AAS}).  

\item \textbf{Type 4.} Almost additive sequences $\cF_4$ with more than one equilibrium measure. They are always physically equivalent to some additive sequence $(S_n\psi)_{n \in \N}$ with $\psi \in \widetilde{WB}/Hof$ and, consequently, never physically equivalent to any additive sequence $(S_n\xi)_{n \in \N}$ with $\xi \in WB$. 
\end{itemize}

\begin{figure}[h]
\centering
\subfigure{\includegraphics[width=14.7cm]{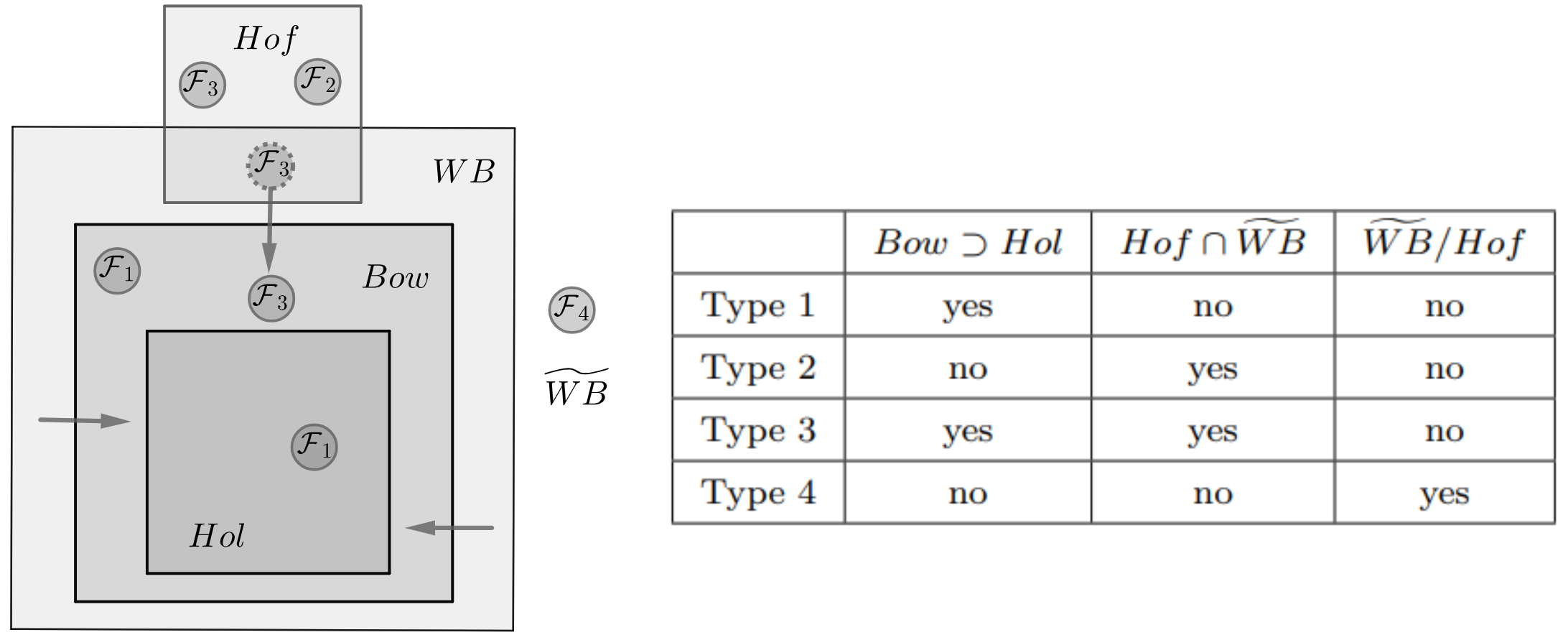}}
\caption{The different types of almost additive sequences and their physical equivalence relations with the space of continuous functions.}\label{AAS}
\end{figure}

Now let us show how one can construct examples for the types 1, 3 and 4. Our general simple strategy to search examples of asymptotically and almost additive sequences is to start with a fixed function $\phi$ in some "appropriate" already known subset of continuous functions and from that to build the nonadditive sequence physically equivalent to $(S_n\phi)_{n \in \N}$. \\

\textbf{Examples of Types 1, 3 and 4.}
Sequences of Type 1 are, for instance, the ones given by Proposition \ref{QB-G}, which also include the H\"older continuous case. Notice that those sequences are always physically equivalent to some almost additive sequence generated by a quasi-Bernoulli measure.

Sequences of Type 3 are constructed in the following way. Letting $\phi \in Bow$ and $\xi \in Wcob/Bow$, we have $\psi:= \phi + \xi \in Hof\cap WB$. Arguing as in the proof of Theorem~\ref{NHD}, for each $n \in \N$ there exists a continuous function $f_n \in Hol$ such that $\|f_n - S_n\psi\|_{\infty}\le 1$ (by the density of H\"older continuous functions). Since $\psi \notin Bow$, the almost additive sequence $\cF = (f_n)_{n \in \N}$ does not have bounded variation. Moreover, $\cF$ is almost additive by Lemma~\ref{aas}. On the other hand, $\psi \in WB$ implies that $(S_n\psi)_{n \in \N}$ is physically equivalent to some additive sequence $(S_nf)_{n \in \N}$ with $f \in Bow$. Then,
\[
\lim_{n \to \infty}\frac{1}{n}\|f_n - S_n f\|_{\infty} \le \lim_{n \to \infty}\frac{1}{n}\|f_n - S_n\psi\|_{\infty} + \lim_{n \to \infty}\frac{1}{n}\|S_n\psi - S_n f\|_{\infty} = 0.
\]
We note that, by Proposition \ref{GBS}, the unique equilibrium measure for $\cF$ does not satisfy the Gibbs property with respect to $\cF$ but is, nevertheless, Gibbs with respect to some Bowen function and, consequently, also quasi-Bernoulli. As in Type 1, these sequences are always physically equivalent to some almost additive sequence generated by a quasi-Bernoulli measure, which is a sequence solely composed of locally constant functions. 

For the other case of Type 3 sequences, we consider a function $\psi \in Hof \cap \widetilde{WB}$ and proceed as before to find an almost additive sequence $\cG = (g_n)_{n \in \N}$ of H\"older continuous functions such that $\|g_n - S_n\psi\|_{\infty} \le 1$ for all $n \in \N$. In particular,  $\cG$ does not have bounded variation, is physically equivalent to $(S_n\psi)_{n \in \N}$ and not physically equivalent to any additive sequence generated by a Bowen function. One more time, Proposition \ref{GBS} guarantees that the unique equilibrium measure for $\cG$ does not satisfy the Gibbs property with respect to $\cG$. In this case, the equilibrium measure still might be weak Gibbs with respect to some continuous function. 

In order to show an example of Type 4, we just start with some continuous function $\psi: \Sigma^{\N} \to \R$ with more than one equilibrium measure. It is clear that $\psi \in \widetilde{WB}/Hof$. Following in the very same manner as for Type 3, we are able to find almost additive sequences without bounded variation and which are physically equivalent to $(S_n\psi)_{n \in \N}$, as desired.  \qed 

Before discussing almost additive sequences of Type 2, let us settle \textbf{Question D}, which is regarding regularity of asymptotically additive sequences. 

We recall that a measure $\mu$ on $\Sigma^{\N}$ (not necessarily invariant) is said to be \emph{weak Gibbs} with respect to a function $\phi: \Sigma^{\N} \to \R$ when there exists a sequence $(K_n)_{n \in \N} \subset [1, \infty)$ with $\log K_n /n = 0$ and such that
\[
K_n^{-1} \le \frac{\mu_n(C_{j_1 \ldots j_n})}{\exp(-nP_{\sigma}(\phi) + S_n\phi(x))} \le K_n
\]
for all $x \in C_{j_1 \ldots j_n}$ and $n \ge 1$. The same definition is also naturally extended to the nonadditive case (see for example \cite{Bar06} and \cite{IY17}). \\

\textbf{Asymptotically additive case.}
Fix a continuous Hofbauer function $\psi \in Hof \cap~\widetilde{WB}$. Since every almost additive sequence admits a (not necessarily invariant) weak Gibbs measure (\cite{Bar06}), in particular, the function $\psi$ admits some weak Gibbs measure $\nu$ on $\Sigma^{\N}$, that is, there exists a sequence $(K_n)_{n \in \N} \subset [1, \infty)$ with $\log K_n /n = 0$ such that 
\[
|\log\nu(C_{i_1\ldots i_n}(x)) + nP(\psi) - S_n\psi(x)| \le \log K_n \quad \textrm{for all $x \in \Sigma^{\N}$ and $n \ge 1$},
\]
which readily implies that the sequence $\cG := (f_n + nP(\psi))_{n \in \N}$ is asymptotically additive and also physically equivalent to $(S_n\psi)_{n \in \N}$, where $f_n(x) := \log \nu(C_{i_1 \ldots i_n}(x))$ for all $x \in \Sigma^{\N}$ and $n \in \N$. Notice that $\cG$ has a unique equilibrium measure and satisfies the Walters property (i.e. also has bounded variation). Moreover, since $\psi \notin WB$, there is no additive sequence generated by a Bowen function which is physically equivalent to $\cG$. In particular, $\cG$ is not physically equivalent to any additive sequence generated by a locally constant function. 

Proceeding as before (see the proof of Theorem \ref{NHD}), for each $\alpha > 0$ there exists a sequence $\cH_{\alpha} := (h^{\alpha}_n)_{n \in \N}$ of H\"older continuous functions such that $\|h^{\alpha}_n - f_n - nP(\psi)\|_{\infty} \le \alpha$ for all $n \in \N$. This implies that $\cH_{\alpha}$ has bounded variation and 
\[
\lim_{n \to \infty}\frac{1}{n}\|h^{\alpha}_n - S_n\psi\|_{\infty} = 0 \quad \textrm{with $\psi \in Hof \cap \widetilde{WB}$}.
\] 
That is, for each $\alpha > 0$ the sequence $\cH_{\alpha}$ is asymptotically additive, has bounded variation, admits a unique equilibrium measure but is not physically equivalent to any additive sequence generated by a Bowen function. In particular, the sequences $\cH_{\alpha}$ are not physically equivalent to any additive sequence generated by a H\"older continuous function. 

These examples demonstrate that \textbf{Question D} cannot be affirmatively answered.
\qed \\

Finally, let us consider the second type of almost additive sequences of potentials. \\

\textbf{Examples of Type 2 (?)} We note that \emph{there exists an almost additive sequence of Type 2 if and only if there exists a quasi-Bernoulli measure not Gibbs with respect to any continuous function.} In fact, suppose $\cF$ is a sequence of Type 2, that is, $\cF$ has bounded variation and has the unique equilibrium measure $\mu$ not Gibbs with respect to any continuous function. Since $\cF$ is Gibbs with respect to $\cF$, Proposition \ref{GQB} guarantees that $\mu$ is a quasi-Bernoulli measure.  

Conversely, considering $\mu$ a quasi-Bernoulli measure, the sequence $\cF = (f_n)_{n \in \N}$ generated by $\mu$ is almost additive and has bounded variation. Since we are assuming $\mu$ not Gibbs with respect to any continuous function, by Proposition \ref{GBS} and Corollary \ref{EQ}, the sequence $\cF$ cannot be physically equivalent to any Bowen function. That is, 
\[
\lim_{n \to \infty}\frac{1}{n}\|f_n - S_n\psi\|_{\infty} = 0
\] 
for some $\psi \in \widetilde{WB}$. Since $\cF$ has bounded variation, it follows from Theorem 6 in \cite{Mum06} that $\cF$ has a unique equilibrium measure, which implies that $\psi \in Hof$.  Therefore, we conclude that $\psi \in Hof \cap \widetilde{WB}$ and $\cF$ is indeed a sequence of Type 2, as desired.   

We still don't know about the existence of such sequences in this setup or in any other setup in general.      \qed                   

\subsection{Some open questions and final comments}

As we saw in the previous subsection, the more general regularity problem relating sequences with bounded variation and Bowen functions is still open, and depends on the existence of sequences of Type 2. Then, inspired by \cite{BKM} and based on Proposition \ref{CCH} and Theorem \ref{BKM-H}, we ask the following:

\textbf{Q1:} \emph{is there a Bowen continuous function $A:\Sigma^{\N} \to GL(d,\R)$ generating a matrix cocycle $\cA: \Sigma^{\N} \times \N \to GL(d,\R)$ where the sequence $\cF = (\log \|\cA(x,n)\|)_{n \in \N}$ is almost additive of Type 2 ? For what dimensions $d \ge 2$ this should be possible ?}

A negative answer to \textbf{Q1} would give a concrete and important class of examples where we can reduce the nonadditive setup directly to the additive one, in the spirit of \cite{BKM}. On the other side, a positive answer would fill the only type of sequences that are missing in our proposed classification. \\

Now let us consider again  the left-sided full shift of finite type $\sigma: \Sigma^{\N} \to \Sigma^{\N}$. Denote by $Wal:= Wal(\Sigma^{\N},\sigma)$ the vector space of continuous Walters functions with respect to $\sigma$, and $Wal_{\mathcal{L}}:= Wal_{\mathcal{L}}(\Sigma^{\N},\sigma)$ the vector space of all functions $f \in Wal(\Sigma^{\N},\sigma)$ and such that $\mathcal{L}f = 0$, where $\mathcal{L}$ is the transfer operator given by 
\[
[\mathcal{L}f](y) = \frac{1}{\# \Sigma}\sum_{\sigma(x) = y}f(x) \quad \textrm{for all $y \in \Sigma^{\N}$.}
\]

T. Bousch (\cite{Bou02}) gave a characterization of Walters functions as 
\[
Wal = Cob\oplus Wal_{\mathcal{L}}\oplus \R.
\]
Based on this, we can ask if

\textbf{Q2:} \emph{Do we also have $Bow = WCob \oplus Wal_{\mathcal{L}}\oplus \R$ ? or even $Bow = Wcob + Wal$ ?}\\

Since $Wal$ is a proper subset of $Bow$ (see \cite{Wal07}), it is clear that $Wcob + Wal \subseteq  Wcob + Bow$ (not necessarily a proper inclusion). Then, a negative answer to $\textbf{Q2}$ would guarantee the existence of a function $\psi \in (Wcob + Bow)/(Wcob + Wal)$. In this case, Proposition \ref{QB-G} would immediately give an almost additive sequence $\cF$ satisfying the Walters property and which is physically equivalent to $(S_n\psi)_{n \in \N}$. Hence, $\cF$ would be an example of an almost additive sequence with Walters regularity but not physically equivalent to any additive sequence generated by a Walters function (or any locally constant function in particular). On the other hand, a positive answer to \textbf{Q2} would imply that for any Bowen function $\phi$ there exists a Walter function $\psi$ such that $(S_n\phi)_{n \in \N}$ and $(S_n\psi)_{n \in \N}$ are physically equivalent. Moreover, Theorem \ref{BOU} would give that $(S_n\phi)_{n \in \N}$ is cohomologous to $(S_n\psi)_{n \in \N}$ in the nonadditive sense of definitions \ref{bv}, \ref{CHM} and \ref{SCH}. In this case, regarding the study of additive and nonadditive thermodynamic properties with respect to subshifts of finite type, Bowen regularity and Walters regularity would be actually equivalent. Notice that, since the additive notion of cohomology is much stronger than the nonadditive version introduced in section \ref{coh}, this equivalence between Bowen and Walters functions wouldn't violate the classical Liv\v{s}ic Theorem. \\

\textbf{Acknowledgments.} The first author was partially supported by NSF of China, grant no. 12222110. The second author was partially supported by CNPq of Brazil, under the project with reference 409198/2021-8.

\end{document}